\documentclass[12pt]{amsart}
\usepackage{amssymb, enumerate, mdwlist}

\evensidemargin\paperwidth
\advance\evensidemargin-\textwidth
\advance\oddsidemargin-1in
\evensidemargin\oddsidemargin

\topmargin\paperheight
\advance\topmargin-\textheight
\advance\topmargin-1in

\theoremstyle{plain}
\newtheorem{theorem}{Theorem}[section]
\theoremstyle{plain}
\newtheorem{proposition}[theorem]{Proposition}
\theoremstyle{plain}
\newtheorem{lemma}[theorem]{Lemma}
\theoremstyle{plain}
\newtheorem{corollary}[theorem]{Corollary}
\theoremstyle{plain}
\newtheorem{problem}[theorem]{Problem}
\theoremstyle{plain}
\newtheorem{definition}[theorem]{Definition}
\theoremstyle{plain}

\theoremstyle{remark}
\newtheorem{remark}[theorem]{Remark}
\theoremstyle{remark}

\theoremstyle{remark}

\title
[Expander and local reflexivity]
{A generalization of expander graphs and local reflexivity of uniform Roe algebras}
\author{Hiroki Sako}
\thanks{The author is a Research Fellow of the Japan Society for the Promotion of Science (PD)}
\address
{Research Institute for Mathematical Sciences, Kyoto University, Kyoto 606-8502, Japan}
\email
{sako@kurims.kyoto-u.ac.jp}

\begin{document}

\begin{abstract}
We introduce a generalization of expander graphs, which is called  a weak expander sequence.
It is proved that a uniform Roe algebra of a weak expander sequence is not locally reflexive. 
It follows that uniform Roe algebras of expander graphs are not exact.
We introduce the notion 
of a generalized box space to
discuss box spaces and expander sequences in a unified framework.
Key tools for the proof are amenable traces and measured groupoids associated with generalized box spaces.
\end{abstract}


\maketitle

\section{Introduction}
An expander sequence is a family of finite graphs which are uniformly locally finite but highly connected. 
It has applications to computer sciences, error correcting codes, and networks.
The first explicit example of an expander sequence was constructed by Margulis \cite{MargulisExpander}. It was constructed from a residually finite group with relative property (T).

Expander graphs give important examples in coarse geometry. Coarse geometry is a study of `large scale uniform structure' of a space.
We study features which do not depend on the local structure.
The most fundamental properties for coarse spaces are property A defined by Yu \cite[Definition 2.1]{Yu:CoarseHilbert} and coarse embeddability into a Hilbert space. 
Yu dealt with these two properties in the study of the coarse Baum--Connes conjecture. 
The conjecture states that the geometric K-theory of a metric space and the analytic K-theory are isomorphic. 
Property A implies coarse embeddability and
coarse embeddability implies the conjecture.

An expander sequence does not coarsely embed into a Hilbert space, since its components are highly connected.
In this paper, we introduce a generalization of expander sequence, which is called a sequence of weak expander spaces.
It is proved 
that weak expander spaces do not have property A
(Corollary \ref{CorollaryNotA}).
This means that connectivity of the spaces is high enough to negate property A.

We often analyze a uniformly locally finite coarse space $X$ by its uniform Roe algebra. 
The algebra is a C$^*$-algebra and can be regarded as a natural linear representation of the space $X$.
Property A is equivalent to nuclearity of the uniform Roe algebra $C^*_\mathrm{u}(X)$ (Skandalis, Tu, and Yu \cite[Theorem 5.3]{SkandalisTuYu}).
Nuclearity of C$^*$-algebras can be interpreted as a finite dimensional approximation property 
(Choi--Effros \cite[Theorem 3.1]{ChoiEffrosNuclearCPAP}, Kirchberg \cite{KirchbergNuclear}). 
In this paper, we deal with another approximation property, called local reflexivity (\cite[Section 5]{EffrosHaagerup}). 
It is much weaker than nuclearity.
Because a weak expander sequence does not have property A,
a uniform Roe algebra of a weak expander sequence is not nuclear. 
The goal of this paper is to show much stronger negation.

\begin{theorem}
Let $X = \bigsqcup_{m = 1}^\infty X_m$ be a sequence of weak expander spaces.
Then the uniform Roe algebra $C^*_\mathrm{u}(X)$ is not locally reflexive.
\end{theorem}
Since exactness implies local reflexivity \cite{KirchbergExactUHF}, we have the following corollary.
\begin{corollary}
A uniform Roe algebra 
of a sequence of expander graphs is not exact.
\end{corollary}

In Section \ref{SectionGenBox}, we prepare 
the notion of a generalized box space,
which is a special kind of a coarse space.
In Section \ref{SectionLocalReflexivity}, we review the definition of a uniform Roe algebras and local reflexivity of C$^*$-algebras.
In Section \ref{SectionGroupoid}, 
we construct a topological groupoid associated to a coarse space. The groupoid is different from that in \cite{SkandalisTuYu}. Its topology is generated by countably many compact and open subsets.
In Section \ref{SectionProof}, it is proved that the groupoid has an invariant measure, if the space is a generalized box space.
If the uniform Roe algebra is locally reflexive, the measured groupoid has a F\o lner property. 
This is a key for the proof of the main theorem. 
In the last section, we make comments on uniform local amenability (ULA) defined by 
Brodzki, Niblo, \v{S}pakula, Willett, and Wright \cite[Definition 2.2]{BNSWW}.
Definition of weak expander sequence is related to ULA.

\section{Generalized box space and weak expander sequence}
\label{SectionGenBox}
\subsection{Coarse space}
We prepare several notations related to coarse geometry.
See Roe's lecture note \cite[Chapter 2]{RoeLectureNote} for details.
Let $X$ be a set. 
For subsets $T, T_1, T_2 \subseteq X^2$, we define the inverse $T^{-1}$ and the product $T_1 \circ T_2$
as follows:
\begin{eqnarray*}
T^{-1} &=& \{(x, y) \in X^2 \ |\ (y,x) \in T\},\\
T_1 \circ T_2 &=& \{ (x, y) \in X^2 \ 
|\ \textrm{there exists\ } z \in X \textrm{\ such that\ }
(x, z) \in T_1, (z, y) \in T_2 \}.
\end{eqnarray*}
Denote by $T^{\circ n}$ the $n$-th power 
$T \circ T \circ \cdots \circ T$.
For a subset $Y \subseteq X$ and $T \subseteq X^2$,
let $T[Y]$ be a  set defined by
\begin{eqnarray*}
T[Y] = \{x \in X\ | \ \textrm{there exists\ } y \in Y \textrm{\ such that\ } (x, y) \in T \}.
\end{eqnarray*}
For a one-point set $\{x\}$, we simply write $T[x] = T[\{ x \}]$. 
A subset $F \subseteq X$ is called a $T$-bounded set if there exists $x \in X$ such that $F \subseteq T[x]$.

\begin{definition}
[Definition 2.3 in \cite{RoeLectureNote}]
Let $X$ be a set and let
$\mathcal{C}$ be a family of subsets of $X^2$.
The pair $(X, \mathcal{C})$ is said to be a coarse space if it satisfies the following:
\begin{itemize}
\item
The diagonal subset $\Delta_X \subseteq X^2$ is an element of $\mathcal{C}$.
\item
If $T_1 \subseteq T_2$ and $T_2 \in \mathcal{C}$, then $T_1 \in \mathcal{C}$.
\item
If $T \in \mathcal{C}$, then $T^{-1} \in \mathcal{C}$.
\item
If $T_1, T_2 \in \mathcal{C}$, then $T_1 \circ T_2 \in \mathcal{C}$.
\item
If $T_1, T_2 \in \mathcal{C}$, then $T_1 \cup T_2 \in \mathcal{C}$.
\end{itemize}
\end{definition}
Two elements $x, y \in X$ are said to be connected if $\{(x, y)\} \in \mathcal{C}$. 
Sometimes, a coarse space is assumed to be connected.
In this paper, connectivity is not required.
For the coarse space $(X, \mathcal{C})$, elements of $\mathcal{C}$ are called controlled sets (or entourages).

A metric space $(X, d)$ is naturally equipped with a coarse structure $\mathcal{C}$ defined by
\begin{eqnarray*}
\mathcal{C} = \{ T \subseteq X^2 \ |\ d \mathrm{\ is\ bounded\ on\ } T \}.
\end{eqnarray*}
Other metrics on $X$ can define the same coarse structure.
In coarse geometry, we focus on features which only depends on the coarse structure $\mathcal{C}$.

\begin{definition}
A coarse space $(X, \mathcal{C})$ is said to be uniformly locally finite
if every controlled set $T \in \mathcal{C}$ satisfies 
$\sup_{x \in X} \sharp(T[x]) < \infty$.
\end{definition}

Let us observe what uniform local finiteness means in the case of metric spaces.
For a metric space $(X, d)$, 
we use the following notations:
\begin{eqnarray*}
N_R(Y) &=& \{x \in X \ |\ d(x, Y) \le R \}, \\
\partial_R (Y) &=& N_R(Y) \setminus Y.
\end{eqnarray*}
A metric space $(X, d)$ is uniformly locally finite if and only if for every distance $R > 0$, 
all the balls $N_R(x)$ of radius $R$ are finite and the numbers of their elements are uniformly bounded, namely,
$\mathrm{sup}_{x \in X} \sharp(N_R(x)) < \infty$.
In many references, metric spaces with this property are said to have bounded geometry.

The other typical example of a coarse space arises from groups and group actions. 
Let $G$ be a discrete group. 
A set $X$ equipped with a $G$-action naturally has a coarse structure. 
For a finite subset $K \subseteq G$, define $\Gamma_K \subseteq X^2$ by $T_K = \{ (g x, x) \ |\ g \in K, x \in X\}$.
We say that a subset $T \subseteq X \times X$ is controlled if
there exists a finite subset $K \subseteq G$ such that $T \subseteq T_K$.
This coarse structure on $X$ is uniformly locally finite.

\subsection{Generalized box space}
The term `box space' was introduced in 
\cite[Definition 11.24]{RoeLectureNote}.
It is a sequence of finite quotient groups of a residually finite group $G$. 
Let $H_1 \supseteq H_2 \supseteq \cdots$ be finite index normal subgroups of $G$ 
whose intersection $\bigcap_{m=1}^\infty H_m$ is $\{1_G\}$.
A box space is the disjoint union 
$\square G = G / H_1 \bigsqcup G/H_2 \bigsqcup \cdots$ 
of finite quotient groups. The natural left translation action of $G$ gives a uniformly locally finite coarse structure on $\square G$.

Box spaces give interesting examples related to property A as follows:
\begin{itemize}
\item
Guentner observed that $\square G$ has property A if and only if $G$ is amenable (\cite[Proposition 11.39]{RoeLectureNote}).
\item
It has been already pointed out by Willett that the uniform Roe algebra $C^*_\mathrm{u}( \square G )$ is not even exact if $G$ is not amenable (see the last sentence of \cite{ArzhantsevaGuentnerSpakula}). 
This is a conclusion of the argument of \cite[Proposition 3.7.11]{OzawaBook}.
\item
For an appropriate choice of finite index normal subgroups of the free group $F_2$, 
the box space $\square F_2$ coarsely embeds into a Hilbert Space, 
although it does not have property  A (Arzhantseva, Guentner, and \v{S}pakula \cite{ArzhantsevaGuentnerSpakula}).
\item
Let $G$ be a residually finite group with Kazhdan's property (T).
The box space $\square G$ is an expander sequence
(see e.g., \cite[Section 6.1]{BekkaDelaHarpeValette}).
\end{itemize}


The construction of box spaces relies on discrete groups.
In this paper, 
we generalize the notion of a box space in the context of coarse geometry and investigate its operator algebraic feature.

\begin{definition}\label{DefinitionGBS}
A uniformly locally finite coarse space $(X, \mathcal{C})$ is called a generalized box space,
if there exists a disjoint decomposition $X = \bigsqcup_{m = 1}^\infty X_m$ satisfying
the following conditions:
\begin{itemize}
\item
Every subset $X_m$ is finite and not empty,
\item
Every controlled set $T \in \mathcal{C}$ is a subset of
$\bigsqcup_{m=1}^\infty X_m \times X_m$.
\end{itemize}
\end{definition}
The second condition means that the components $\{X_m\}$ are mutually disjoint in the sense of coarse geometry.
This notion allows us to discuss box spaces and expander sequences in the same framework.

\subsection{Weak expander spaces}
\label{SubsectionWeakExpander}
We introduce the notion of weak expander spaces.
\begin{definition}\label{DefinitionWeakExpander}
A generalized box space $X = \bigsqcup_{m = 1}^\infty X_m$ is called a sequence of
weak expander spaces if 
there exist a controlled set $T$ and 
a positive number $c$ satisfying the following condition:
for every controlled set $\widetilde{T}$,
\begin{eqnarray*}
\liminf_{m \rightarrow \infty} \left(
\inf \left\{ \frac{ \sharp(T[Y]) }{ \sharp(Y) } \ \colon\ 
\emptyset \neq Y \subseteq X_m, Y \textrm{\ is\ a } 
\widetilde{T} \textrm{-bounded\ set} \right\} \right) > 1 + c.
\end{eqnarray*}
\end{definition}

Consider the case that $T$ contains the diagonal set.
Then the set $T[Y]$ is a kind of $Y$'s neighborhood and we regard the term $\sharp(T[Y]) / \sharp (Y)$ 
as the expansion of $Y$.
The condition in Definition \ref{DefinitionWeakExpander}
means that 
the expansion of $Y$ 
is uniformly greater than $1 + c$ if
$m$ is large enough and if $Y$ is a $\widetilde{T}$-bounded set. We first observe that the cardinality of $X_m$ diverges, 
like an expander sequence.

\begin{lemma}\label{LemmaDiverge}
If a generalized box space $X = \bigsqcup_{m = 1}^\infty X_m$
is a sequence of weak expander spaces, then $\lim_{m \rightarrow \infty} \sharp(X_m) = \infty$.
\end{lemma}
\begin{proof}
Take a controlled set $T$ and a positive number $c$ in Definition \ref{DefinitionWeakExpander}. 
We may assume that $T$ contains the diagonal set $\Delta_X$.
Let $n$ be a large natural number. If $m$ is large enough, then we have
\begin{eqnarray*}
\inf \left\{ \frac{ \sharp(T[Y]) }{ \sharp(Y) } \ \colon\ 
\emptyset \neq Y \subseteq X_m, Y \textrm{\ is\ a } 
T^{\circ n} \textrm{-bounded\ set} \right\} > 1.
\end{eqnarray*}
Take an element $x \in X_m$.
By the inequality, we have
\begin{eqnarray*}
1 
= \sharp(\{x\})
< \sharp(T[x])
< \sharp(T \circ T [x])
< \cdots
< \sharp(T^{\circ n} [x])
\end{eqnarray*}
and $n < \sharp(T^{\circ n} [x]) \le \sharp(X_m)$.
It follows that $\sharp(X_m)$ is grater than $n$ for large $m$.
\end{proof}

For a box space of a residually finite group $G$, being a weak expander sequence means non-amenability of $G$.

\begin{lemma}
Let $G$ be a finitely generated residually finite group.
Let $H_1 \supseteq H_2 \supseteq \cdots$ be a sequence of 
finite index normal subgroups of $G$ which satisfies 
$\bigcap_{m = 1}^\infty H_m = \{1_G\}$.
Then the following conditions are equivalent:
\begin{enumerate}
\item\label{LemmaBoxConditionNonAme}
The group $G$ is not amenable,
\item\label{LemmaBoxPropertyNotA}
The box space 
$\square G = \bigsqcup_{m = 1}^\infty G / H_m$ does not
have property A,
\item\label{LemmaBoxPropertyWES}
The box space $\square G$ is a sequence of weak expander spaces.
\end{enumerate}
\end{lemma}

\begin{proof}
By \cite[Proposition 11.39]{RoeLectureNote},
conditions (\ref{LemmaBoxConditionNonAme}) and
(\ref{LemmaBoxPropertyNotA}) are equivalent.
We prove the equivalence of conditions (\ref{LemmaBoxConditionNonAme}) and (\ref{LemmaBoxPropertyWES}).
Let $K$ be a symmetric generating set of $G$ containing $1_G$.
Define a controlled set $T_K$ by the action of $K$, i.e.,
$T_K = \{ (g x, x) \ 
|\ g \in K, x \in \square G \}$.
The controlled set $T_K^{\circ n}$ corresponds to the action of $K^n$, namely,
\begin{eqnarray*}
T_K^{\circ n} = \{ (g_1 g_2 \cdots g_n x, x) \ 
|\ g_1, g_2, \cdots, g_n \in K, x \in \square G \}.
\end{eqnarray*}

Suppose that $G$ is amenable.
For every controlled set $T$ of $\square G$, there exists a natural number $n$ such that $T \subseteq T_K^{\circ n}$.
For an arbitrary positive number $\epsilon$, there exists an $\epsilon$-F\o lner set $F \subseteq G$ for the action of $K^n$.
We mean by $\epsilon$-F\o lner that
$\sharp(F) < \infty$ and that 
$\sharp(K^n F) < (1 + \epsilon) \sharp (F)$.
If $m$ is large enough, then the quotient map $q_m \colon G \rightarrow G / H_m$ is injective on the subset $K^n F \subseteq G$.
For large $m$, we have
\begin{eqnarray*}
     \sharp(T[q_m(F)]) 
\le \sharp(T_K^{\circ n}[q_m(F)]) 
 =  \sharp(q_m(K^n F)) 
 <  (1 + \epsilon) \sharp (q_m(F)).
\end{eqnarray*}
This implies the inequality
\begin{eqnarray*}
\liminf_{m \rightarrow \infty} \left(
\frac{ \sharp(T[q_m(F)]) }{ \sharp(q_m(F)) } \right) < 1 + \epsilon.
\end{eqnarray*}
There exists a controlled set $\widetilde{T}$ such 
that for every $m$
the finite subset $q_m(F)$ is $\widetilde{T}$-bounded.
Indeed, there exists a natural number $l$ such that 
$F \subseteq K^l$. By the inclusion 
$q_m(F) \subseteq T_K^{\circ l} [1_{G/H_m}]$,
the subsets $q_m(F)$ are $T_K^{\circ l}$-bounded.
It follows that the box space $\square G$ is not a sequence of weak expander spaces.

Conversely, suppose that $\square G$ is not a sequence of weak expander spaces. 
For every positive number $\epsilon$, there exists a controlled set $\widetilde{T}$ satisfying the following inequality:
\begin{eqnarray*}
\liminf_{m \rightarrow \infty} \left(
\inf \left\{ \frac{ \sharp(T_K[Y]) }{ \sharp(Y) } \ \colon\ 
\emptyset \neq Y \subseteq G / H_m, Y \textrm{\ is\ a } 
\widetilde{T} \textrm{-bounded\ set} \right\} \right) 
< 1 + \epsilon.
\end{eqnarray*}
We may assume that $\widetilde{T} = T_K^{\circ n}$. 
For large enough $m$, 
the quotient map $q_m \colon G \rightarrow G / H_m$ 
is injective on 
$K^{n + 1} \subseteq G$. Take $m$ such that $q_m |_{K^{n + 1}}$ is injective and that 
\begin{eqnarray*}
\inf \left\{ \frac{ \sharp(T_K[Y]) }{ \sharp(Y) } \ \colon\ 
\emptyset \neq Y \subseteq G / H_m, Y \textrm{\ is\ a } 
T_K^{\circ n} \textrm{-bounded\ set} \right\} < 1 + \epsilon.
\end{eqnarray*}
Choose a $T_K^{\circ n}$-bounded subset $Y \subseteq G / H_m$ satisfying $\sharp(T_K[Y]) < (1 + \epsilon) \sharp(Y)$.
Replacing $Y$ with a right translation of $Y$, 
we may assume that $1_{G/H_m} \in Y$.
Since $T_K[Y]$ is $T_K^{\circ (n+1)}$-bounded,
$T_K[Y]$ is included in the image of $K^{n + 1}$. 
The inverse image $F$ of $Y$ under the injective map $q_m |_{K^{n + 1}}$ satisfies the inequality
\begin{eqnarray*}
\sharp(K F) < (1 + \epsilon) \sharp(F).
\end{eqnarray*}
It turns out that there exists an $\epsilon$-F\o lner set with respect to the generating set $K$.
It follows that $G$ is amenable.
\end{proof}

\subsection{Expander and weak expander}
\label{SubsectionExpander}
Let $\{(X_m, d)\}_{m \in \mathbb{N}}$ 
be a sequence of finite metric spaces.
The disjoint union $\bigsqcup_{m = 1}^\infty X_m$ is equipped with a coarse structure defined by
\begin{eqnarray*}
\mathcal{C} =
\left\{\left.
T \subseteq \bigsqcup X_m \times X_m \ \right| \ 
d \mathrm{\ is\ bounded\ on\ } T \right\}.
\end{eqnarray*}
Suppose that 
$\left( \bigsqcup_{m = 1}^\infty X_m, \mathcal{C} \right)$
is a uniformly locally finite coarse space.
Definition \ref{DefinitionWeakExpander}
can be rephrased in terms of the metric $d$.
The space $\bigsqcup_{m = 1}^\infty (X_m, d)$ is 
a sequence of weak expander spaces if and only if
there exist positive numbers $c$ and $R$ such that
for every positive number $S$,
\begin{eqnarray*}
\liminf_{m \rightarrow \infty} \left(
\inf \left\{ \frac{ \sharp(N_R(Y)) }
{ \sharp(Y) } \ \colon\ 
\emptyset \neq Y \subseteq X_m, \mathrm{diam}(Y) \le S
\right\} \right) > 1 + c.
\end{eqnarray*}

To get a better understanding about a weak expander sequence, take an expanding sequence of subsets $Y, N_R(Y), N_R(N_R(Y)), \cdots \subseteq X_m$. 
Let us think that 
the sequence describes how information spreads.
The numbers $\sharp(Y)$, $\sharp(N_R(Y))$, $\cdots$ indicate how many points the information reaches.
The inequality 
\begin{eqnarray*}
\inf \left\{ \sharp(N_R(Y)) / \sharp(Y)  \ \colon\ 
\emptyset \neq Y \subseteq X_m, \mathrm{diam}(Y) \le S
\right\} > 1 + c.
\end{eqnarray*}
means that the number of the points 
increases exponentially ($\ge O((1 + c)^n))$), until the diameter of the area exceeds $S$. We can make the constant $S$ bigger, by choosing a larger component of $\bigsqcup X_m$.


The notion of weak expander spaces is a generalization of expander graphs. 
We recall the definition of expander graphs. 
For a connected graph $(X, E)$ with a vertex set $X$ 
and an edge set $E$, we define a metric on $X$ by
\begin{eqnarray*}
d(x, y) = \mathrm{min} \{d \in \mathbb{N} \ 
|\ x = z_0, z_1, \cdots, z_d = y \in X, 
z_{j-1} \textrm{\ and\ } z_j \textrm{\ are\ connected}\}.
\end{eqnarray*}
For a subset $Y \subseteq X$, $N_1(Y)$ is the collection of all the vertices which are connected with $Y$ by edges.
A sequence $\{(X_m, E_m)\}_{m \in \mathbb{N}}$ 
of finite graphs is called an expander sequence, if the following conditions hold:
\begin{enumerate}
\item
The graphs $(X_m, E_m)$ are connected,
\item
There exists a natural number $D$ such that the graphs $(X_m, E_m)$ are all $D$-regular. Namely, for every $m$ and $x \in X_m$, $\sharp(\partial_1(x)) = D$,
\item\label{DefinitionExpanderCardinality}
$\lim_m \sharp (X_m) = \infty$,
\item\label{DefinitionExpanderExpansion}
There exists a constant $c > 0$ such that  
for every subset
$Y \subseteq X_m$ satisfying $\sharp(Y) \le \sharp(X_m) / 2$,
the inequality
$\sharp(N_1(Y))  > (1 + c) \sharp(Y)$ holds.
\end{enumerate} 
Fix distance $S > 0$. For a subset $Y \subseteq X_m$, if 
$\mathrm{diam} (Y) < S$ and if $m$ is large enough,
then $\sharp(Y) \le \sharp (X_m) / 2$ by condition (\ref{DefinitionExpanderCardinality}).
By condition (\ref{DefinitionExpanderExpansion}), 
an expander sequence is a sequence of weak expander spaces.

A sequence of expander graphs is characterized in two ways.
In the above description, 
we used a kind of negation of the F\o lner condition.
A sequence of weak expander spaces is a generalization in this context.
A sequence of expander graphs can also be characterized by means of Poincar{\' e} inequality.
Ostrovskii \cite[Theorem 2.4]{Ostrovskii} and Tessera \cite{TesseraPoincare} independently characterized
metric spaces which are not coarsely embeddable into 
Hilbert spaces and $L^1$-spaces.
They made use of weaker versions of Poincar{\' e} inequality.
Tessera also studied non-embeddability into uniformly convex Banach spaces and ${\rm CAT}(0)$ spaces
(\cite[Corollary 17]{TesseraPoincare}). 
Our generalization does not imply 
this kind of non-embeddability.
Indeed, the box space $\square F_2$ of Arzhantseva, Guentner, and \v{S}pakula \cite{ArzhantsevaGuentnerSpakula} is a sequence of weak expanders and coarsely embeds into a Hilbert space.

\section{Preliminary on operator algebra}
\label{SectionLocalReflexivity}

\subsection{Translation algebra and uniform Roe algebra}
Let us recall the definition of a uniform Roe algebra. 
We regard the algebra as a natural linear representation of a coarse space.
A bounded linear operator $a$ on $\ell_2 X$ is said to have finite propagation if its matrix coefficient is located on a controlled set. 
More precisely, there exists a controlled set $T_a \subseteq X \times X$ 
such that $\langle a \delta_y, \delta_x \rangle = 0, (x, y) \in (X \times X) \setminus T_a$. 
The collection of all the operators with finite propagation is called the translation algebra $\mathcal{A}^\infty(X)$. 
The uniform Roe algebra $C^*_\mathrm{u}(X)$ is the operator norm closure of $\mathcal{A}^\infty(X)$. 
A coarse geometric property of $X$ sometimes implies an operator algebraic property of $C^*_\mathrm{u}(X)$. The typical example is the following:
\begin{theorem}
[Theorem 5.3 in \cite{SkandalisTuYu}]
Let $X$ be a uniformly locally finite metric space. The space $X$ has property A if and only if $C^*_\mathrm{u}(X)$ is nuclear.
\end{theorem}

\subsection{Local reflexivity}
In this paper, we consider the following properties for C$^*$-algebras:
Nuclearity,
exactness,
and local reflexivity.
Nuclearity implies exactness. Conclusions in Kirchberg \cite{KirchbergExactUHF} show that
exactness implies local reflexivity.
They are all related to minimal tensor products 
between C$^*$-algebras. 
The following is the definition of local reflexivity.
\begin{definition}[Section 5 of Effros--Haagerup \cite{EffrosHaagerup}]
A C$^*$-algebra $B$ is said to be locally reflexive if for every finite dimensional operator system $V \subseteq B^{**}$, 
there exists a net of contractive completely positive maps $\Phi_j \colon V \rightarrow B$ 
which converges to $\mathrm{id}_V$ in the point-ultraweak topology.
\end{definition}
Instead of the definition, we only use the following features
of local reflexivity.
\begin{proposition}
[Theorem 3.2, Theorem 5.1, and Proposition 5.3 of 
\cite{EffrosHaagerup}]
\label{PropositionLocalSplit}\mbox{}
\begin{itemize}
\item
A C$^*$-subalgebra of a locally reflexive C$^*$-algebra is also locally reflexive.
\item
If $B$ is locally reflexive and $J$ is an ideal of $B$, 
then the exact sequence 
$0 \rightarrow J \rightarrow B \rightarrow B/J \rightarrow 0$ 
locally splits.
Namely, for every finite dimensional operator system 
$V \subseteq B / J$, 
there exists a unital completely positive map $\Phi \colon V \rightarrow B$ such that $\Phi(a) + J = a$, $a \in V$. 
\end{itemize}
\end{proposition}
For a neat description on local reflexivity, 
the reader is referred to the book \cite[Chapter 9]{OzawaBook} by Brown--Ozawa.

\subsection{Amenable trace}
\label{SubsectionAmeTr}
We recall the definition of amenable trace.
\begin{definition}\label{DefinitionAmeTr}
Let $B \subseteq \mathbb{B}(\mathcal{H})$ be a unital C$^*$-algebra. 
A tracial state $\theta$ on $B$ is said to be amenable if there exists a state $\rho$ on $\mathbb{B}(\mathcal{H})$ 
such that $\rho |_B = \theta$ and
$\rho(uau^*) = \rho(a)$ for every 
$a \in \mathbb{B}(\mathcal{H})$ and 
unitary $u \in B$.
\end{definition}
By Arveson's extension theorem, this definition does not depend on the choice of the faithful representation.
Kirchberg defined the notion of a liftable trace (\cite[Definition 3.1]{KirchbergT}) and proved that it is equivalent to amenability in the above sense 
(\cite[Proposition 3.2]{KirchbergT}).
Instead of Definition \ref{DefinitionAmeTr}, 
we use other equivalent conditions.
Let $\pi_\theta \colon B \rightarrow \mathbb{B}(\mathcal{H}_\theta)$ be the GNS-representation 
for the tracial state $\theta$. 
\begin{theorem}[Theorem 6.2.7 of \cite{OzawaBook}]
\label{TheoremAmenableTrace}
Let $\theta$ be a tracial state of a unital C$^*$-algebra $B$.
The following statements are equivalent:
\begin{enumerate}
\item
The tracial state $\theta$ is amenable,
\item\label{TheoremAmenableTraceConditionNet}
There exists a net of unital completely positive maps $\Psi_j \colon B \rightarrow \mathbb{M}(n(j))$ satisfying that
$\theta(b) = \lim \mathrm{tr} \circ \Psi_j(b)$ and $\lim \| \Psi_j(ab) - \Psi_j (a) \Psi_j (b) \|_{\mathrm{tr}, 2} = 0$,
for every $a, b \in B$,
where $\mathrm{tr}$ stands for the normalized trace of $\mathbb{M}(n(j))$,
\item
For every faithful representation $B \subseteq \mathbb{B} (\mathcal{H})$, 
there exists a unital completely positive map $\Psi \colon \mathbb{B} (\mathcal{H}) \rightarrow \pi_\theta(B)^{\prime\prime}$ 
whose restriction to $B$ is $\pi_\theta$.
\end{enumerate}
\end{theorem}

Amenability of $\theta$ passes to that on the algebra $\pi_\theta(B)$, if $B$ is locally reflexive.

\begin{proposition}\label{PropositionQuotient}
Let $\theta$ be an amenable trace of a unital C$^*$-algebra $B$. If $B$ is locally reflexive, 
then the state $\tau = \langle \cdot \xi_\theta, \xi_\theta \rangle$ is an amenable trace of $\pi_\theta(B)$.
\end{proposition}

\begin{proof}
Since $B$ is locally reflexive, by Proposition \ref{PropositionLocalSplit}, 
the sequence 
$0 \rightarrow \mathrm{Ker}(\pi_\theta) \rightarrow B \rightarrow \pi_\theta(B) \rightarrow 0$ is a locally split extension. 
By \cite[Proposition 6.3.5 (4)]{OzawaBook}, 
$\tau$ is an amenable trace of $\pi_\theta(B)$.
\end{proof}

For a generalized box space $X = \bigsqcup_{m = 1}^\infty X_m$, 
the uniform Roe algebra $C^*_\mathrm{u} (X)$ and its unital subalgebra $B$ have amenable traces. 
Let $\Psi_m \colon C^*_\mathrm{u} (X) \rightarrow \mathbb{B}(\ell _2 X_m) = \mathbb{M}(\sharp(X_m), \mathbb{C})$ be 
the unital completely positive map by compression. 
Since $X_m$ is isolated in the coarse space $X$, $\Psi_m$ is a $*$-homomorphism. 
The composition $\theta_m = \mathrm{tr} \circ \Psi_m$ is a trace on $C^*_\mathrm{u} (X)$.
Let $\theta$ be an accumulation point of $\{\theta_m\}$ in the state space of $C^*_\mathrm{u} (X)$. 
By Theorem \ref{TheoremAmenableTrace} (\ref{TheoremAmenableTraceConditionNet}), this is an amenable trace of $C^*_\mathrm{u} (X)$ and of its unital C$^*$-subalgebra $B$.

\section{Construction of groupoid associated to label}
\label{SectionGroupoid}
Groupoid is an algebraical object equipped with partially defined associative product, a space of units, and inverse.
We often consider additional structures such as measure and topology. 
Measured groupoids are related to constructions of von Neumann algebras 
and ergodic theory of group actions.
Topological groupoids serve as fundamental tools for operator K-theory and coarse geometry.
For the axioms of groupoids, the reader is referred to Renault's book \cite[Chapter 1]{BookRenault}.

Skandalis, Tu, and Yu described a coarse space $X$ 
in terms of topological groupoid in \cite[Section 3]{SkandalisTuYu}. 
Their groupoid was constructed from Stone--\v{C}ech compactification of $X$. 
For the proof of our theorem, this groupoid is too big. 
We need separability of translation C$^*$-algebra and the groupoid must be second countable.
In this section, we first construct a C$^*$-algebra, which is associated to a finite collection of partial bijections.
We describe the algebra as a representation of a groupoid, 
whose topology is generated by countably many closed and open subsets.

Let $X$ be a uniformly locally finite coarse space.
For a controlled set $T$, there exists a label $L$ in the following sense. 
\begin{definition}
Let $T$ be a controlled set containing the diagonal subset 
$\Delta_X \subseteq T \subseteq X^2$.
Label on $T$ is a family $L = \{\Delta_X = \phi(0), \phi(1), \phi(2), \cdots, \phi(k) \}$ of subsets of $T$ satisfying
the following conditions:
\begin{itemize}
\item
$T = \Delta_X \bigcup \left(\bigcup_{i = 1}^k \phi(i) \right)$,
\item
Each $\phi(i)$ gives a bijection from a subset of $X$ to a subset of $X$. Namely,
if $(z, x), (z, y) \in \phi(i)$, then $x = y$,
if $(x, z), (y, z) \in \phi(i)$, then $x = y$.
\end{itemize}
\end{definition}
When $(x, y) \in X^2$ is an element of $\phi(i)$, 
we regard $x$ as the image of $y$ 
with respect to the map $\phi(i)$.
Then we write $x = [\phi(i)](y)$.
Let $(X, T)^{(2)} \subseteq X^2$ be the set of the pairs which are connected in the coarse structure generated by $T$.
Namely, $(X, T)^{(2)}$ is the set 
$\bigcup_{n = 1}^\infty (T \cup T^{-1})^{\circ n}$.
The set $(X, T)^{(2)}$ naturally has a groupoid structure by
\begin{itemize}
\item
range map: $\mathrm{Image}_X \colon (x, y) \mapsto x$,
\item
source map: $\mathrm{Dom}_X \colon (x, y) \mapsto y$,
\item
product: $(x, y)(y, z) = (x, z)$,
\item
inverse: $(x, y)^{-1} = (y, x)$.
\end{itemize}
We often identify the space of units $\Delta_X$ with $X$.
The first goal of this section is 
Theorem \ref{TheoremGroupoid}.
The theorem states that there exists a `good' groupoid $\Gamma$ which encodes the partial bijections $\phi(i)$.

\subsection{Translation C$^*$-algebra associated to label}
\label{SubsectionC$^*$-SubalgebraLabel}
For $i = -k, -k+1, \cdots , -1$, define $\phi(i)$ by the inverse 
$\phi(-i)^{-1}$.
Let $I$ be the index set $\{ -k, -k+1, \cdots , 0, \cdots, k \}$.
For $i \in I$, we denote by $v(i) : \ell_2 X \rightarrow \ell_2 X$ the partial isometry defined by
\begin{eqnarray*}
[v(i)](\delta_y)
=
\left\{
\begin{array}{cc}
\delta_{[\phi(i)](y)}, & \quad y \in \mathrm{Dom}_X (\phi(i)),\\
0,                         & \quad y \notin \mathrm{Dom}_X (\phi(i)).
\end{array}
\right.
\end{eqnarray*}
The operator $v(i)$ is an element of the uniform Roe algebra 
$C^*_\mathrm{u}(X) \subseteq \mathbb{B}(\ell_2 X)$.
The operators satisfy the relation $v(-i) = v(i)^*$.
Let $I^*$ be the index set $\bigsqcup_{n = 1}^\infty I^n$. For 
\begin{eqnarray*}
g = (g(1), g(2), \cdots, g(n)) \in I^n \subseteq I^*,
\end{eqnarray*}
we define a partial bijection $\phi(g)$ by the composition 
$\phi(g(1)) \circ \phi(g(2)) \circ \cdots \circ \phi(g(n))$,
restricting the domain in the obvious way.
The composition $\circ$ is identical to the product as controlled sets.

Let $E \colon \mathbb{B}(\ell_2 X) \rightarrow \ell_\infty X$ be the conditional expectation
onto $\ell_\infty X$. 
We define operators $v(g), p(g)$ in $C^*_\mathrm{u}(X)$ as follows:
\begin{itemize}
\item
a partial isometry $v(g) = v(g(1)) v(g(2)) \cdots v(g(n))$,
\item
a projection $p(g) = E(v(g))$.
\end{itemize}
Let $B = B_L$ be the unital C$^*$-subalgebra of $C^*_\mathrm{u}(X)$ 
generated by $\{ v(g) \ |\ g \in I^* \}$ and  
$\{ p(g) \ |\ g \in I^* \}$.
Let $A$ be 
the commutative C$^*$-algebra generated 
by $E(B) \subseteq \ell_\infty X$.

\begin{lemma}\label{Lemma;First}
The algebra $A$ is generated 
by at most countably many projections.
The algebra $B$ contains $A$.
The algebra $A$ is equal to $E(B)$.
\end{lemma}

\begin{proof}
Let $B_0$ be the linear span of the set
\begin{eqnarray*}
\{ v(g_1) p(h_1) v(g_2) p(h_2) \cdots v(g_n) p(h_n) \ | \ g_1, h_1, g_2, h_2, \cdots, g_n, h_n \in I^* \}.
\end{eqnarray*}
The set $B_0$ is closed under multiplication and involution.
The algebra $B$ is the closure of $B_0$.
Note that all the partial isometries $v(g_l)$ correspond to partial bijections of $X$
and that $p(h_l) \in \ell_\infty X$ is a characteristic function of a subset of $X$.
It turns out that the image of the conditional expectation 
$E(v(g_1) p(h_1) \cdots v(g_n) p(h_n))$
has to be a projection. 
It follows that $A$ is generated by these countably many projections.

Define $g \in I^*$ by $(g_1(1), g_1(2), \cdots, g_2(1), g_2(2), \cdots, \cdots, g_n(1), g_n(2), \cdots)$.
Define $v$ by $v(g_1) p(h_1) v(g_2) p(h_2) \cdots v(g_n) p(h_n) \in B_0$.
The matrix coefficients of $v$ are $1$
only on a subset of $\phi(g) \subseteq X \times X$. 
It follows that $v^* v \le v(g)^* v(g)$.
Since $v$ corresponds to a partial bijection on $X$,
the projection $v^* v$ is in $\ell_\infty X$.
By the equality $v = v(g) v^* v$, we have $E(v) = p(g) v^* v$.
It follows that $E(B_0)$ is a subset of $B$.
We conclude that $A$ is a subalgebra of $B$
and equal to $E(B)$.
\end{proof}

\subsection{Groupoid associated to label}
\label{SubsectionGroupoidLabel}
Let $\Omega$ be a Gelfand spectrum of the unital commutative C$^*$-algebra $A$.
We often identify $A$ and $C(\Omega)$.
The projection $p \in C(\Omega)$ is a characteristic function
of a closed and open subset $\mathrm{supp}_\Omega(p) \subseteq \Omega$. 
We identify $A p$ and $C(\mathrm{supp}_\Omega(p))$.
By Lemma \ref{Lemma;First},
the topology of $\Omega$ is generated by 
at most countably many closed and open subsets.
For $g \in I^*$, $\mathrm{Ad} (v(g))$ gives an isomorphism from $A v(g)^* v(g)$
to $A v(g) v(g)^*$.
This isomorphism defines 
a homeomorphism $\psi_g \colon \mathrm{supp}_\Omega(v(g)^* v(g)) \rightarrow \mathrm{supp}_\Omega(v(g) v(g)^*)$ by 
\begin{eqnarray*}
[v(g) a v(g)^* ] (\psi_g (\alpha)) = a (\alpha), 
\quad 
a \in A v(g)^* v(g),
\alpha \in \mathrm{supp}_\Omega(v(g)^* v(g)).
\end{eqnarray*}
We denote by $\mathrm{Dom}_\Omega(\psi_g)$,
$\mathrm{Image}_\Omega(\psi_g) \subseteq \Omega$ the domain and the image of $\psi_g$.

We define a set $\Sigma(g)$ by 
\begin{eqnarray*}
\Sigma(g) = \{ (\psi_g (\alpha), g, \alpha) \ 
|\ g \in I^*, \alpha \in \mathrm{Dom}_\Omega(\psi_g)\}
\end{eqnarray*}
and define its topology by the identification
$\Sigma(g) \cong \mathrm{Dom}_\Omega(\psi_g) \subseteq \Omega$.
We define $\Sigma$ by the disjoint union 
$\Sigma = \bigsqcup_{g \in I^*} \Sigma(g)$
and introduce its topology by the product 
of the compact Hausdorff topology and the discrete topology.
We define an relation $\sim$ 
on $\Sigma$ by the following:
two elements $(\psi_g(\alpha), g, \alpha)$ and $(\psi_h(\beta), h, \beta)$ are said to be equivalent
if $\alpha = \beta$ and $\alpha \in \mathrm{supp}_\Omega (E(v(h)^* v(g)))$.

\begin{remark}\label{RemarkPsi}
If $(\psi_g(\alpha), g, \alpha) \sim (\psi_h(\beta), h, \beta)$,
then $\psi_h(\beta)$ is equal to $\psi_g(\alpha)$.
Since the function $E(v(h)^* v(g)) \in A$ is 
a characteristic function of the set 
$\{x \in X \ | \ [\phi(g)](x) = [\phi(h)](x) \} \subseteq X$, $\mathrm{Ad} (v(g))$ and $\mathrm{Ad} (v(h))$ 
are identical on $A E(v(h)^* v(g))$. 
Since $\alpha \in \Omega$ is in 
$\mathrm{supp}_\Omega (E(v(h)^* v(g)))$, we have
$\psi_h(\beta) = \psi_h(\alpha) = \psi_g(\alpha)$.
The converse does not hold in general.
\end{remark}
\begin{lemma}
The relation $\sim$ on $\Sigma$ is an equivalence relation.
\end{lemma}

\begin{proof}
For the proof of transitivity, suppose that 
$(\psi_g(\alpha), g, \alpha) \sim (\psi_h(\beta), h, \beta)$ and
$(\psi_h(\beta), h, \beta) \sim 
(\psi_g(\alpha^\prime), g^\prime, \alpha^\prime)$.
Then $\alpha$ is equal to $\alpha^\prime$ and included in the supports of 
$E(v(h)^* v(g))$ and 
$E(v(g^\prime)^* v(h))$.
Regarding the operators $E(v(g^\prime)^* v(h))$ and 
$E(v(h)^* v(g))$ as elements in $\ell_\infty X$, 
we have the inequality $E(v(g^\prime)^* v(h)) E(v(h)^* v(g)) \le E(v(g^\prime)^* v(g))$.
It follows that $\alpha \in \mathrm{supp}_\Omega (E(v(g^\prime)^* v(g)))$.
We have 
$(\psi_g(\alpha), g, \alpha) \sim 
(\psi_g(\alpha^\prime), g^\prime, \alpha^\prime)$.
\end{proof}

Let $\Gamma$ be the quotient topological space $\Sigma / \sim$ and $q \colon \Sigma \rightarrow \Gamma$ be the quotient map.
\begin{lemma}\label{LemmaTopologyOfGamma}
\begin{enumerate}
\item\label{Hausdorff}
The space $\Gamma$ is Hausdorff.
\item
For every $g \in I^*$, 
$q(\Sigma(g))$ is a closed and open subset of $\Gamma$.
\item
For every $g \in I^*$, 
$q |_{\Sigma(g)} \colon \Sigma(g) \rightarrow q(\Sigma(g))$ is a homeomorphism. 
\end{enumerate}
\end{lemma}
\begin{proof}
Let $\gamma_1, \gamma_2$ be two distinct elements of $\Gamma$.
Take representatives 
$(\psi_g(\alpha_1), g, \alpha_1)$, $(\psi_h(\alpha_2), h, \alpha_2) \in \Sigma$ of $\gamma_1$ and $\gamma_2$, respectively.
We first suppose that $\alpha_1 \neq \alpha_2 \in \Omega$.
Since $\Omega$ is a Hausdorff space,
there exist disjoint open subsets $\Omega_1, \Omega_2 \subseteq \Omega$ such that $\alpha_l \in \Omega_l$
for $l = 1, 2$.
The subsets 
\begin{eqnarray*}
O_1 &=& 
\bigcup_{k \in I^*} \{ q (\psi_k(\beta), k, \beta) \ 
|\ \beta \in \mathrm{Dom}_\Omega(\psi_k) \cap \Omega_1 \},\\
O_2 &=& 
\bigcup_{k \in I^*} \{ q (\psi_k(\beta), k, \beta) \ 
|\ \beta \in \mathrm{Dom}_\Omega(\psi_k) \cap \Omega_2 \}
\end{eqnarray*}
are open in $\Sigma$.
Indeed, these inverse images 
\begin{eqnarray*}
q^{-1} (O_1) &=&
\bigcup_{k \in I^*} \{ (\psi_k(\beta), k, \beta) \ 
|\ \beta \in \mathrm{Dom}_\Omega(\psi_k) \cap \Omega_1 \},\\
q^{-1} (O_2) &=& 
\bigcup_{k \in I^*} \{ (\psi_k(\beta), k, \beta) \ 
|\ \beta \in \mathrm{Dom}_\Omega(\psi_k) \cap \Omega_2 \} 
\end{eqnarray*} 
are open in $\Sigma$.
The subset $O_1$ contains $\gamma_1$ and $O_2$ contains $\gamma_2$. These open subsets are disjoint.

Suppose that $\alpha_1 = \alpha_2$ and 
$\alpha_1 \notin \mathrm{supp}_\Omega E(v(h)^* v(g))$.
Define a projection $p \in A$ by $p = 1 - E(v(h)^* v(g))$.
Define two subsets $O_1, O_2$ of $\Gamma$ by
\begin{eqnarray*}
O_1        
&=& \{ q (\psi_g(\beta), g, \beta) \ 
|\ \beta \in \mathrm{Dom}_\Omega(\psi_g) \cap \mathrm{supp}_\Omega (p) \},\\
O_2        
&=& \{ q (\psi_h(\beta), h, \beta) \ 
|\ \beta \in \mathrm{Dom}_\Omega (\psi_h) \cap \mathrm{supp}_\Omega (p) \}.
\end{eqnarray*}
The inverse images are 
\begin{eqnarray*}
q^{-1} (O_1)
&=& \bigcup_{k \in I^*} 
\{ (\psi_k(\beta), k, \beta) \ 
|\ k \in I^*, \beta \in 
\mathrm{supp}_\Omega (E(v(k)^* v(g)))
\cap \mathrm{supp}_\Omega (p) \},\\
q^{-1} (O_2)        
&=& \bigcup_{k \in I^*} 
\{ (\psi_k(\beta), k, \beta) \ 
|\ k \in I^*, \beta \in 
\mathrm{supp}_\Omega(E(v(k)^* v(h))) 
\cap \mathrm{supp}_\Omega (p) \}.
\end{eqnarray*}
Since they are open in $\Sigma$, $O_1$ and $O_2$ are open in $\Gamma$.
Moreover, $q^{-1} (O_1)$ and $q^{-1} (O_2)$ are disjoint, since
$E(v(k)^* v(g)) E(v(k)^* v(h)) p = 0$.
Therefore the open subsets $O_1$ and $O_2$ are also disjoint.
Since $\gamma_1 \in O_1$ and $\gamma_2 \in O_2$, 
$\gamma_1$ and $\gamma_2$ are separated by two open subsets. 
It follows that $\Gamma$ is Hausdorff.

By compactness of $\Sigma(g)$ and continuity of $q$, 
the image $q (\Sigma(g))$ is also compact.
Since $\Gamma$ is Hausdorff,
$q (\Sigma(g))$ is closed.
The inverse image $q^{-1} (q (\Sigma(g)))$ is equal to
\begin{eqnarray*}
\bigcup_{k \in I^*} \{ (\psi_k (\beta), k, \beta) \ 
|\ k \in I^*, \beta \in \mathrm{supp}_\Omega (E(v(k)^* v(g))) \}.
\end{eqnarray*}
This is an open subset of $\Sigma$.
It follows that $q (\Sigma(g))$ is open.

The map $q |_{\Sigma(g)}$ is injective by the definition of the equivalence relation $\sim$.
The map $q |_{\Sigma(g)} \colon \Sigma(g) \rightarrow q(\Sigma(g))$ is a bijective continuous map
from a compact space onto a Hausdorff space. It follows that the map is homeomorphic. 
\end{proof}

We denote by $\Gamma(g)$ the subset $q(\Sigma(g)) \subset \Gamma$.
We describe an element $q (\psi_g(\alpha), g, \alpha)$ of $\Gamma(g)$ as $[\psi_g(\alpha), g, \alpha]$.
Before stating the next proposition, 
fix some notations.
\begin{itemize}
\item
The element $0 \in I^*$ stands for $(0) \in I^1$.
\item
For $g = (g(1), g(2), \cdots, g(n)), h = (h(1), h(2), \cdots, h(m)) \in I^*$, 
the product $g * h$ is defined by $(g(1), g(2), \cdots, g(n), h(1), h(2), \cdots, h(m))$. 
\item
The inverse $g^{-1}$ is defined by 
$(- g(n), - g(n - 1), \cdots, - g(1))$.
\end{itemize}
The following relations in $B$ holds:
$v(0) = 1$, $v(g) v(h) = v(g * h)$, and $v(g)^* = v(g^{-1})$.

We prove that $\Gamma$ naturally has a structure of an \'{e}tale groupoid.
\begin{proposition}
The set $\Gamma$ is an \'etale groupoid with the
space of units
\begin{eqnarray*}
\Omega \cong \Gamma(0) = \{ [\alpha, 0, \alpha] \ |\ \alpha \in \Omega \}
\end{eqnarray*}
equipped with the following well-defined continuous operations:
\begin{enumerate}
\item
{range map or target map},
$t \colon \Gamma\rightarrow \Omega \colon [\psi_g (\alpha), g, \alpha] 
\mapsto \psi_g(\alpha)$,
\item
{source map},
$s \colon \Gamma\rightarrow \Omega \colon [\psi_g (\alpha), g, \alpha] \mapsto \alpha$,
\item
{product}
$\colon 
\Gamma^{(2)} \rightarrow \Gamma \colon 
([\psi_g (\alpha), g, \alpha], [\psi_h (\beta), h, \beta]) \mapsto [\psi_g (\alpha), g * h, \beta]$,\\
where $\Gamma^{(2)}$ is 
$\{ ([\psi_g (\alpha), g, \alpha], [\psi_h (\beta), h, \beta]) \in \Gamma \times \Gamma\ 
|\ \alpha = \psi_h (\beta)\} $, 
\item
{inverse}
$\colon \Gamma\rightarrow \Gamma
\colon [\psi_g (\alpha), g, \alpha] \mapsto [\alpha, g^{-1}, \psi_g (\alpha)]$.
\end{enumerate}
\end{proposition}

\begin{proof}
By the definition of $\Gamma$, the source map $s$ is well-defined.
The map is continuous, since $s \circ q \colon \Sigma \rightarrow \Omega$ is continuous.
By Lemma \ref{LemmaTopologyOfGamma} (2), 
$\Gamma(g)$ is an open subset of $\Gamma$.
Since the map $s$ is homeomorphic on $\Gamma(g)$,
$s$ is locally homeomorphic. 

We prove that the product map is well-defined.
We only show that the product map does not depend 
on the choice of representatives of the first entry.
Take an element $[\psi_{g * h}(\beta), g, \psi_h(\beta)], [\psi_h(\beta), h, \beta]) \in \Gamma^{(2)}$
and choose another representative $(\psi_{g * h}(\beta), g^\prime, \psi_h(\beta))$ of $[\psi_{g * h}(\beta), g, \psi_h(\beta)]$.
By Remark \ref{RemarkPsi}, every representative is  of this form.
By the definition, we have $\psi_h(\beta) \in \mathrm{supp}_\Omega (E(v(g^\prime)^* v(g)))$.
It follows that 
\begin{eqnarray*}
\beta 
\in \psi_h^{-1} (\mathrm{supp}_\Omega (E(v(g^\prime)^* v(g))))
&=& \mathrm{supp}_\Omega (v(h)^* E(v(g^\prime)^* v(g)) v(h))\\
&=& \mathrm{supp}_\Omega 
(E(v(h)^* v(g^\prime)^* v(g) v(h)))\\
&=& \mathrm{supp}_\Omega (E(v(g^\prime * h)^* v(g * h))).
\end{eqnarray*}
This implies the equality 
$[\psi_{g * h}(\beta), g * h, \beta] = 
[\psi_{g * h}(\beta), g^\prime * h, \beta]$.
We obtain the claim.

To prove the continuity of the product map, define $\Sigma^{(2)}$ by
\begin{eqnarray*}
\Sigma^{(2)} = (q \times q)^{-1} (\Gamma^{(2)}) =
\{ ((\psi_g (\alpha), g, \alpha), (\psi_h (\beta), h, \beta)) \in \Sigma \times \Sigma \ 
|\ \alpha = \psi_h (\beta) \}.
\end{eqnarray*} 
The map 
$\Sigma^{(2)} \ni ((\psi_g (\alpha), g, \alpha), (\psi_h (\beta), h, \beta)) 
\mapsto (\psi_g (\alpha), g * h, \beta) \in \Sigma$
is continuous. Since the space $\Gamma^{(2)}$ is a quotient of $\Sigma^{(2)}$,
the product map on $\Gamma^{(2)}$ is also continuous.

To prove that the inverse map is well-defined,
take an element $\gamma = [\psi_g(\alpha), g, \alpha]$.
By Remark \ref{RemarkPsi}, every representative of $\gamma$
is of the form $(\psi_g(\alpha), h, \alpha)$. 
We also have 
\begin{eqnarray*}
\psi_h (\beta) 
\in  \psi_h (\mathrm{supp}_\Omega (E(v(h)^* v(g))))
&=&    \mathrm{supp}_\Omega (v(h) E(v(h)^* v(g)) v(h)^*)\\
&=&    \mathrm{supp}_\Omega (E(v(g^{-1})^* v(h^{-1}))).
\end{eqnarray*}
It follows that $[\alpha, g^{-1}, \psi_g(\alpha)] = [\beta, h^{-1}, \psi_h(\beta)]$.
This means that the inverse map is well-defined.
The map is also continuous, since the map 
$\Sigma \rightarrow \Sigma
\colon (\psi_g (\alpha), g, \alpha) \mapsto (\alpha, g^{-1}, \psi_g (\alpha))$ is continuous.
The range map $t$ is also well-defined and continuous, since $t$ is the composition of $s$ and the inverse map.

Simple calculations show that $\Gamma$ satisfy algebraic axioms of groupoid.
\end{proof}

\subsection{Construction of a homomorphism}
\label{SubsectionProofGroupoid}
We finish the proof of Theorem \ref{TheoremGroupoid}.

\begin{theorem}\label{TheoremGroupoid}
Let $T$ be a controlled set on a uniformly locally finite coarse space $X$. Suppose that $T$ contains the diagonal subset $\Delta_X$.
Let $L$ be a label $\{\Delta_X = \phi(0), \phi(1), \phi(2), \cdots, \phi(k) \}$ on $T$.
There exist a groupoid $\Gamma$, a groupoid homomorphism $\Phi^* \colon (X, T)^{(2)} \rightarrow \Gamma$, 
and compact and open subsets $\Gamma(1), \Gamma(2), \cdots, \Gamma(k) \subseteq \Gamma$
which satisfy the following conditions:
\begin{enumerate}[(a)]
\item\label{TheoremGroupoidEtale}
The groupoid $\Gamma$ is \'{e}tale and Hausdorff.
\item\label{TheoremGroupoidClopen}
The space of the units $\Gamma(0)$ is compact, and the topology is generated by at most countably many closed and open subsets.
\item\label{TheoremGroupoidGenerator}
The groupoid $\Gamma$ is generated by $\Gamma(0), \Gamma(1), \cdots, \Gamma(k)$.
\suspend{enumerate}
For $i = 0, 1, 2, \cdots, k$, 
\resume{enumerate}[{[(a)]}]
\item\label{TheoremGroupoidInverseImagek}
The source map $s$ and 
the range map $t$ are injective on $\Gamma(i)$ . 
\item\label{TheoremGroupoidUnits}
The inverse image $(\Phi^*)^{-1} (\Gamma(i))$
is $\phi(i)$.
The image $\Phi^*(\phi(i))$ is dense in $\Gamma(i)$.
\item\label{TheoremGroupoidSurjectivity}
Let $\gamma$ be an element of 
$\Gamma$ and 
let $x$ be an element of $X$.
If $t(\gamma) = \Phi^*(x, x)$, then there exists $y \in X$ such that $\gamma = \Phi^*(x, y)$.
If $s(\gamma) = \Phi^*(x, x)$, then there exists $y \in X$ such that $\gamma = \Phi^*(y ,x)$.
\end{enumerate}
\end{theorem}

We have already shown that our concrete $\Gamma$ satisfies 
(\ref{TheoremGroupoidEtale}), (\ref{TheoremGroupoidClopen}),
and (\ref{TheoremGroupoidInverseImagek}).
Condition (\ref{TheoremGroupoidGenerator}) also holds. Indeed, every element $[\psi_g(\alpha), g, \alpha]$
of $\Gamma$ can be written as
\begin{eqnarray*}
[\psi_{g(1)}(\psi_{(g(2), \cdots, g(n))}(\alpha)), g(1), \psi_{(g(2), \cdots, g(n))}(\alpha)] \cdot \cdots \cdot
[\psi_{g(n)}(\alpha), g(n), \alpha].
\end{eqnarray*}
This is a product of elements in $\Gamma(i)$ and their inverses.

We construct a homomorphism $\Phi^*$ satisfying conditions
(\ref{TheoremGroupoidUnits}) and (\ref{TheoremGroupoidSurjectivity}).
\begin{lemma}\label{LemmaPhi*}
Let $\Gamma$ be the groupoid constructed
in the previous subsections.
Let $(x, y)$ be an element of $(X, T)^{(2)}$.
Suppose that 
$g \in I^*$ satisfies $x = [\phi(g)](y)$. Denote by $\widehat{x} ,\widehat{y} \in \Omega$ the characters of $A$ by the evaluations at $x$ and $y$.
\begin{itemize}
\item
If $g \in I^*$ satisfies $x = [\phi(g)](y)$, then
$[\widehat{x}, g, \widehat{y}]$ is an element of $\Gamma$,
\item
If $g, h \in I^*$ satisfy $x = [\phi(g)](y) = [\phi(h)](y)$, then $[\widehat{x}, g, \widehat{y}] = [\widehat{x}, h, \widehat{y}]$.
\end{itemize}
\end{lemma}

\begin{proof}
For every function $a \in A \subseteq \ell_\infty X$,
we have 
\begin{eqnarray*}
[\psi_g (\widehat{y})] (a) = \widehat{y} (v(g)^* a v(g))
= [v(g)^* a v(g)] (y) = a([\phi(g)](y)) = a(x) = \widehat{x} (a).
\end{eqnarray*}
We obtain the equality $\widehat{x} = \psi_g (\widehat{y})$. The first assertion follows.
The equality $[\phi(g)] (y) = [\phi(h)] (y)$ implies that $y$ is in $\mathrm{supp}_X (E(v(h)^* v(g)))$.
Thus we have $\widehat{y} \in 
\mathrm{supp}_\Omega (E(v(h)^* v(g)))$.
We obtain the second assertion.
\end{proof}

We define a map 
$\Phi^* \colon (X, T)^{(2)} \rightarrow \Gamma$ by 
$\Phi^*(x, y) = [\widehat{x}, g, \widehat{y}]$, where
$g$ is an element of $I^*$ satisfying $[\phi(g)](y) = x$.
Lemma \ref{LemmaPhi*} shows that
$\Phi^*$ is well-defined.
Simple calculations show that $\Phi^*$ is a groupoid homomorphism.

\begin{lemma}\label{LemmaSurjectivity}
If $\gamma \in \Gamma(g)$ and $x \in X$ satisfy $t(\gamma) = \widehat{x}$, 
then there exists $y \in X$ such that 
$(x, y) \in \phi(g)$ and 
$\gamma = [\widehat{x}, g, \widehat{y}]$.
If $s(\gamma) = \widehat{x}$, 
then there exists $y \in X$ such that 
$(y, x) \in \phi(g)$ and 
$\gamma = [\widehat{y}, g, \widehat{x}]$.
\end{lemma}

\begin{proof}
Take an element $\gamma = [\alpha, g, \beta]$ of $\Gamma(g)$.
If $t(\gamma) = \widehat{x}$, then $\alpha = \widehat{x}$ and $\beta = \psi_g^{-1} (\widehat{x})$. 
The equality $\psi_g^{-1} (\widehat{x}) = \widehat{[\phi(g)]^{-1}(x)}$ holds. By letting $y = [\phi(g)]^{-1} (x)$, we get 
the first assertion. 
We obtain the second assertion in the same way.
\end{proof}

We obtain condition (\ref{TheoremGroupoidSurjectivity}).
We next prove that $\Gamma$ satisfies (\ref{TheoremGroupoidUnits}).
Suppose that $\Phi^*(x, y) \in \Gamma(i)$.
Take $g \in I^*$ satisfying $x = [\phi(g)](y)$.
Since $[\widehat{x}, g, \widehat{y}] 
= \Phi^*(x, y) \in \Gamma(i)$, we obtain the equality 
$[\widehat{x}, g, \widehat{y}] 
= [\psi_i (\widehat{y}), i, \widehat{y}]$.
Since $\widehat{y} \in \mathrm{supp}_\Omega(E(v(i)^* v(g)))$,
we have $y \in \mathrm{supp}_X (E(v(i)^* v(g)))$. The equation $x = [\phi(g)](y) = [\phi(i)]y$ follows. 
It follows that the inverse image of $\Gamma(i)$ under $\Phi^*$ is $\phi(i)$.
This is the first half of (\ref{TheoremGroupoidUnits}).
Note that $s$ gives a homeomorphism from $\Gamma(i)$ onto
$\mathrm{Dom}_\Omega (\psi_i)$.
Since the map $\ \widehat{}\ \colon \mathrm{Dom}_X (\phi(i)) \rightarrow \mathrm{Dom}_\Omega (\psi_i)$ has dense image,
so does $\Phi^* \colon \phi(i) \rightarrow \Gamma(i)$. 
We conclude the second half of (\ref{TheoremGroupoidUnits}).
This is the end of the proof of Theorem \ref{TheoremGroupoid}.

\subsection{Translation algebra as a representation of groupoid}
We observe several conclusions of Theorem \ref{TheoremGroupoid}.
We assume that
$\Gamma$, 
$\Phi^* \colon (X, T)^{(2)} \rightarrow \Gamma$, 
and $\Gamma(1)$, $\cdots$, $\Gamma(k)$ 
satisfy the conditions in Theorem \ref{TheoremGroupoid}. 
One does not need to think that they are items constructed in the previous subsections.
We prove that the translation algebra of $X$ gives a representation of $\Gamma$.
For the rest of this paper, we use the following notations.
\begin{itemize}
\item
Let $I$ be the set 
$\{-k, -k +1, \cdots, 0, 1, \cdots, k\}$ and 
$I^*$ be the index set $\bigsqcup_{n = 1}^\infty I^n$.
\item
For negative $i \in I$, define $\Gamma(i)$ by $\Gamma(-i)^{-1}$.
\item
For $g = (g(1), g(2), \cdots, g(n)) \in I^n \subseteq I^*$,
define $\Gamma(g)$ by the product 
\begin{eqnarray*}
\Gamma(g(1)) \cdots \Gamma(g(n))
= \{ \gamma_1 \gamma_2 \cdots \gamma_n \ |\ 
\gamma_i \in \Gamma(g(i)), s(\gamma_i) = t(\gamma_{i + 1}) \}.
\end{eqnarray*}
The set $\Gamma(g)$ is compact and open.
\item
For $g \in I^*$, denote by $\psi_g$ the characteristic function of $\Gamma(g)$. The function $\psi_g$ is continuous and compactly supported. 
\end{itemize}
By condition (\ref{TheoremGroupoidGenerator}), $\{\Gamma(g)\}$ is an open covering of $\Gamma$.

\begin{lemma}\label{LemmaBij}
For any $x \in X$, 
the map $\Phi^*$ defines a bijection from
$(\mathrm{Image}_X)^{-1}(x)$ to $t^{-1} (\Phi^*(x, x))$.
The map $\Phi^*$ gives a bijection from 
$(\mathrm{Dom}_X)^{-1}(x)$ to $s^{-1} (\Phi^*(x, x))$.
\end{lemma}

\begin{proof}
We only prove that $\Phi^* \colon (\mathrm{Image}_X)^{-1}(x) \rightarrow t^{-1} (\Phi^*(x, x))$ is bijective.
Take arbitrary elements $(x, y), (x, z) \in (\mathrm{Image}_X)^{-1}(x)$ and suppose that $\Phi^*(x, y) = \Phi^*(x, z)$. 
Then we have $\Phi^*(y, z) = \Phi^*(x, y)^{-1} \Phi^*(x, z)$ 
$\in \Gamma(0)$. By condition (\ref{TheoremGroupoidUnits})
for $\Gamma(0)$, we have $(y, z) \in \phi(0)$ and $y = z$.
We conclude that $\Phi^* \colon (\mathrm{Image}_X)^{-1}(x) \rightarrow t^{-1} (\Phi^*(x, x))$ is injective.
By condition (\ref{TheoremGroupoidSurjectivity}), 
the map is surjective.
\end{proof}

\begin{lemma}\label{LemmaInverseImage}
Let $\Lambda_1, \Lambda_2$ be subsets of $\Gamma$.
Define $\Lambda_1 \Lambda_2$ by $
\{\gamma_1 \gamma_2 \ |\ \gamma_l \in \Gamma_l, s(\gamma_1) = t(\gamma_2)\}$.
Then we have $(\Phi^*)^{-1}(\Lambda_1 \Lambda_2) = 
(\Phi^*)^{-1}(\Lambda_1) \circ (\Phi^*)^{-1}(\Lambda_2)$.
\end{lemma}

\begin{proof}
It suffices to show that $(\Phi^*)^{-1} (\Lambda_1 \Lambda_2) \subseteq (\Phi^*)^{-1}(\Lambda_1) \circ (\Phi^*)^{-1}(\Lambda_2)$. 
Suppose that $(x, z) \in (\Phi^*)^{-1} (\Lambda_1 \Lambda_2)$. Then there exist $\gamma_1 \in \Lambda_1$ and $\gamma_2 \in \Lambda_2$ such that $\Phi^*(x, z) = \gamma_1 \gamma_2$. 
By the equation $\Phi^*(x, x) = t(\gamma_1)$ and
by condition (\ref{TheoremGroupoidSurjectivity}),
there exists $y \in X$ such that $\Phi^*(x, y) = \gamma_1$. 
We also have $\Phi^*(y, z) = \Phi^*(y, x) \Phi^*(x, z) = \gamma_1^{-1} (\gamma_1 \gamma_2) = \gamma_2$.
It follows that $(x, z) = (x, y)(y, z) \in (\Phi^*)^{-1}(\Lambda_1) \circ (\Phi^*)^{-1}(\Lambda_2)$.
\end{proof}

Denote by $C_c(\Gamma)$ the set of all the continuous functions on $\Gamma$ whose supports are compact.
The space $C_c(\Gamma)$ is a $*$-algebra equipped with product and involution
\begin{eqnarray*}
[f f^\prime](\gamma_0) &=& \sum_
{(\gamma, \gamma^{\prime}) \in \Gamma^{(2)}, \  
\gamma \gamma^\prime = \gamma_0}
f(\gamma) f^\prime(\gamma^\prime),\\
f^*(\gamma_0) &=& \overline{f(\gamma_0^{-1})},
\quad f, f^\prime \in C_c(\Gamma), \quad \gamma_0 \in \Gamma.
\end{eqnarray*}
The characteristic function $\psi_0$ of $\Gamma(0)$ is 
the unit of $C_c(\Gamma)$.

Define a map $\Phi$ from $C_c(\Gamma)$ to the set of functions on $(X, T)^{(2)}$ by the pull back of $\Phi^*$.
Namely, $\Phi$ is given by the equality $[\Phi(f)](x, y) = f(\Phi^*(x, y))$. 
We regard $\Phi(f) = [[\Phi(f)](x, y)]_{(x, y) \in (X, T)^{(2)}}$
as an infinite matrix and as an operator on $\ell_2 X$.

\begin{lemma}\label{LemmaEquationOfPhi}
The map $\Phi$ gives a unital $*$-homomorphism from $C_c(\Gamma)$ to the translation algebra $\mathcal{A}^\infty(X)$.
\end{lemma}

\begin{proof}
Since $\{\Gamma(g)\}$ covers $\Gamma$,
$\bigcup_g C(\Gamma(g))$ spans $C_c(\Gamma)$.
For every continuous function $f$ on $\Gamma(g)$, $\Phi(f) = f \circ \Phi^*$ is supported on a graph of a partial bijection $\phi(g) \subseteq (X, T)^{(2)}$, by condition (\ref{TheoremGroupoidUnits}).
The matrix $\Phi(f)$ defines 
a bounded operator with finite propagation. 
It follows that that the image of $\Phi$ is contained in $\mathcal{A}^\infty(X)$.

For $f, f^\prime \in C_c(\Gamma)$ and 
$(x, y) \in (X, T)^{(2)}$, we have
\begin{eqnarray*}
\langle \Phi(f) \Phi(f^\prime) \delta_y, \delta_x \rangle
&=&
\sum_{(x, z) \in (\mathrm{Image}_X)^{-1}(x)}
\langle \Phi(f) \delta_z, \delta_x \rangle
\langle \Phi(f^\prime) \delta_y, \delta_z \rangle \\
&=&
\sum_{(x, z) \in (\mathrm{Image}_X)^{-1}(x)}
f (\Phi^*(x, z)) f^\prime (\Phi^*(z, y)).
\end{eqnarray*}
Since $\Phi^* \colon (\mathrm{Image}_X)^{-1}(x)
\rightarrow t^{-1}(\Phi^*(x, x))$ is a bijection 
(Lemma \ref{LemmaBij}), we have
\begin{eqnarray*}
\langle \Phi(f) \Phi(f^\prime) \delta_y, \delta_x \rangle
&=&
\sum_{\gamma \in t^{-1}(\Phi^*(x, x))}
f (\gamma) f^\prime (\gamma^{-1} \Phi^*(x, y))\\
&=&
[f f^\prime] (\Phi^*(x, y))\\
&=&
\langle \Phi(f f^\prime) \delta_y, \delta_x \rangle.
\end{eqnarray*}
It follows that $\Phi$ is multiplicative. Since $\Phi^*$ is compatible with the inverse maps, $\Phi$ is a $*$-homomorphism.
By condition (\ref{TheoremGroupoidUnits}) for $\Gamma(0)$, the matrix coefficients of $\Phi(\psi_0)$ are $1$ on the diagonal set  and $0$ on the compliment. It follows that $\Phi$ is unital.
\end{proof}

\subsection{Haar system on the groupoid $\Gamma$}

For an unit $\alpha \in \Gamma(0)$ of $\Gamma$,
let $\nu^\alpha$ be the counting measure on the set
$t^{-1}(\alpha)$.
The family of measures $\nu = \{\nu^\alpha\}$ satisfies the axioms of continuous Haar system:
\begin{itemize}
\item
The support of $\nu^\alpha$ is $t^{-1}(\alpha)$,
\item
For $f \in C_c(\Gamma)$, the map $\Gamma(0) \ni \alpha \mapsto \int_{\gamma \in t^{-1}(\alpha)} f d \nu^\alpha \in \mathbb{C}$
is continuous,
\item
For every $\gamma \in \Gamma$ and every $f \in C_c(\Gamma)$, $\int_{\gamma^\prime} 
f(\gamma \gamma^\prime) d \nu^{s(\gamma)} 
= \int_{\gamma^\prime} f(\gamma^\prime) d \nu^{t(\gamma)}$.
\end{itemize}
For Haar systems, the reader is referred to Renault's book \cite[Chapter I. 2]{BookRenault}.
Let $\mu$ be a measure on $\Gamma(0)$.
Consider a measure $\mu \circ \nu$ 
on $\Gamma$ defined by
\begin{eqnarray*}
\int_{\gamma \in \Gamma} 
f(\gamma) d \mu \circ \nu(\gamma) = 
\int_{\alpha \in \Gamma(0)} 
\left( \int_{\gamma \in t^{-1} (\alpha) } 
f(\gamma) d \nu^\alpha(\gamma) \right)
d \mu(\alpha), \quad f \in C_c(\Gamma).
\end{eqnarray*}

Let $w = \sum_{x \in X} w_x \delta_x$ 
be a probability measure on $X$. 
Denote by $c^x$ be the counting measure of the set
$(\mathrm{Image}_X)^{-1} (x)$.
We define the measure $w \circ c$ by 
\begin{eqnarray*}
w \circ c (Z) = 
\sum_{x \in X} w_x c^x
((\mathrm{Image}_X)^{-1}(x) \cap Z), \quad
Z \subseteq (X, T)^{(2)}.
\end{eqnarray*}

\begin{lemma}\label{LemmaPushForward}
Suppose that $\mu$ is the push forward of $w$ with respect to $\Phi^* \colon X \cong \Delta_X \rightarrow \Gamma(0)$.
Then the measure $\mu \circ \nu$ is equal to the push forward of $w \circ c$.
\end{lemma}

\begin{proof}
This lemma immediately follows from 
Lemma \ref{LemmaBij}.
\end{proof}

\section{Proof of the main theorem}
\label{SectionProof}

From now on, we focus on box spaces.
For a box space of a residually finite group, every component has group structures.
Components of a generalized box space do not have group structures. 
Instead of groups, we exploit the groupoid in Theorem \ref{TheoremGroupoid}. 
The groupoid naturally has an invariant measure. 

\subsection{Generalized box space and measured groupoid}
\label{SubsectionMeasuredGroupoid}
Let $X = \bigsqcup_{m = 1}^\infty X_m$ be a generalized box space and let $T$ be
a controlled set of $X$. Let $L$ be a label on $T$.
Suppose that a groupoid $\Gamma$, a homomorphism $\Phi^* \colon (X, T)^{(2)} \rightarrow \Gamma$ and closed and open subsets $\Gamma(1), \cdots, \Gamma(k) \subseteq \Gamma$ satisfy Theorem
\ref{TheoremGroupoid}.
Let $w_m$ be the normalized counting measure on the component $X_m$, i.e.,
$w_m(Y) = \sharp(Y) / \sharp(X_m), \ Y \subseteq X_m$.
Define $\mu_m$ by the push forward of $w_m$.
Let $\mu$ be an accumulation point of 
$\{\mu_m\}_{m \in \mathbb{N}} \in \mathrm{Prob}(\Gamma(0))$.
We call $\mu$ a `limit measure.'

\begin{lemma}\label{LemmaInvInv}
The measure $\mu \circ \nu$ on $\Gamma$ is invariant under the inverse map.
Namely,
the equation $\mu \circ \nu(\Lambda) = \mu \circ \nu(\Lambda^{-1})$ holds for every Borel subset $\Lambda \subseteq \Gamma$.
\end{lemma}

\begin{proof}
The measure $w_m \circ c$ on $(X, T)^{(2)}$ is
invariant under the inverse map, since the measure is a scalar multiple of the counting measure of $(X_m, T)^{(2)}$.
By Lemma \ref{LemmaPushForward}, $\mu_m \circ \nu$ is
also invariant under the inverse map.
The measure $\mu \circ \nu$ is an accumulation point
of $\{ \mu_m \circ \nu \}$ on every $\Gamma(g)$. It follows that $\mu \circ \nu$ is also an invariant measure.
\end{proof}

An element $\psi \in C_c (\Gamma)$ acts on the Hilbert space $L^2(\Gamma, \mu \circ \nu)$ by 
a bounded operator $\lambda(f)$ as follows:
\begin{eqnarray*}
[\lambda(\psi) \xi](\gamma) &=& \sum_
{(\gamma^\prime, \gamma^{\prime\prime}) \in \Gamma^{(2)}, \  
\gamma^\prime \gamma^{\prime\prime} = \gamma}
\psi(\gamma^\prime) \xi(\gamma^{\prime\prime}),
\quad \xi \in L^2(\Gamma, \mu \circ \nu), \gamma \in \Gamma.
\end{eqnarray*}
The map $\lambda \colon C_c(\Gamma) \rightarrow \mathbb{B}(\mathcal{H})$ is a $*$-representation of 
$C_c(\Gamma)$ and
called the regular representation of the
measured groupoid $(\Gamma, \nu, \mu)$. 
The closure of the image $\lambda(C_c(\Gamma))$
is called the reduced C$^*$-algebra 
and written as $C^*_{\mathrm{red}}(\Gamma, \nu, \mu)$.
We denote by $W^*(\Gamma, \nu, \mu)$ the von Neumann algebra
generated by $C^*_{\mathrm{red}}(\Gamma, \nu, \mu)$.
Denote by $\psi_0$ be the characteristic function of $\Gamma(0)$ in $L^2(\Gamma, \mu \circ \nu)$.
\begin{lemma}
\begin{enumerate}
\item
The vector $\psi_0$ is cyclic for $C^*_{\mathrm{red}} (\Gamma, \nu, \mu)$.
\item
The state $\tau = \langle \cdot \psi_0, \psi_0 \rangle_{L^2(\Gamma, \mu \circ \nu)}$ on
$W^* (\Gamma, \nu, \mu)$ is tracial.
\item
The state $\tau$ is faithful.
\end{enumerate}
\end{lemma}

\begin{proof}
Because $\lambda(C_c(\Gamma))\psi_0$
consists of compactly supported continuous functions,
the first statement follows.
For every $f \in C_c(\Gamma)$, we have 
$\tau (\lambda(f)^* \lambda(f)) = 
\| \lambda(f) \psi_0 \|^2_2 = 
\| f \|_2^2$.
Since the measure $\mu \circ \nu$ is invariant under the inverse map, the equality 
$\| f \|_2^2 = 
\| f^* \|_2^2$ follows.
We have
 $\tau (\lambda(f)^* \lambda(f)) 
= \tau (\lambda(f) \lambda(f)^*)$.
It follows that $\tau$ has the trace property.

Define an anti-linear isometry $J$ on 
$L^2(\Gamma, \mu \circ \nu)$ by
\begin{eqnarray*}
[J \xi](\gamma) = \overline{\xi (\gamma^{-1})}, \quad
\xi \in L^2(\Gamma, \mu \circ \nu), \gamma \in \Gamma.
\end{eqnarray*}
A simple calculation shows the following equality:
\begin{eqnarray*}
[\lambda(f_1) J \lambda(f_2) J \xi](\gamma) 
&=&
\sum_{(\gamma_1, \gamma^\prime, \gamma_2) \in \Gamma^{(3)}, 
\gamma_1 \gamma^\prime \gamma_2 = \gamma}
f_1(\gamma_1) \xi(\gamma^\prime) \overline{f_2(\gamma_2)}\\
&=&
[J \lambda(f_2) J \lambda(f_1) \xi](\gamma). 
\end{eqnarray*}
It follows that $J \lambda(f_2) J$ belongs to  $W^* (\Gamma, \nu, \mu)^\prime$. 
The space $J \lambda(C_c(\Gamma)) J \psi_0$ is equal to
$\lambda(C_c(\Gamma)) \psi_0$ and dense in $L^2(\Gamma, \mu \circ \nu)$.
Since $J \lambda(C_c(\Gamma)) J \subseteq 
W^* (\Gamma, \nu, \mu)^\prime$,
the vector $\psi_0$ is
cyclic for $W^* (\Gamma, \nu, \mu)^\prime$.
It follows that $\psi_0$ is a separating vector
for $W^* (\Gamma, \nu, \mu)$.
\end{proof}

The C$^*$-algebra $C(\Gamma(0))$ is a unital $*$-subalgebra of $C_c(\Gamma)$.
Let $A$ be the algebra $\Phi(C(\Gamma(0))) \subseteq \ell_\infty X$.
By condition (\ref{TheoremGroupoidUnits}) in Theorem
\ref{TheoremGroupoid}, $A$ is isomorphic to $C(\Gamma(0))$.
We denote by $B$ the closure of $\Phi(C_c(\Gamma))$.
Let $E$ be the conditional expectation from $\mathbb{B}(\ell_\infty X)$ onto $\ell_\infty X$.
Again by condition (\ref{TheoremGroupoidUnits}), 
we have the equation $\Phi(f |_{\Gamma(0)}) = E(\Phi(f))$, 
$f \in C_c(\Gamma)$.
It follows that $A$ is a subalgebra of $B$ and equals to $E(B)$.
The limit measure $\mu$ on $\Gamma(0)$ gives 
a state $\theta$ on $A$. We extend the state $\theta$ to $B$ by $\theta \circ E \in B^*$.
Let $(\pi_\theta, \mathcal{H}_\theta, \xi_\theta)$ be 
the GNS triple of the state $\theta \in B^*$, namely,
\begin{eqnarray*}
\langle \pi_\theta(b) \xi_\theta, \xi_\theta \rangle = \theta(b), \quad b \in B, \quad
\overline{ \pi_\theta(B) \xi_\theta} = \mathcal{H}_\theta.
\end{eqnarray*}
This construction of the trace $\theta \in B^*$ is the same as 
that in Subsection \ref{SubsectionAmeTr}. 
By the last paragraph of Subsection \ref{SubsectionAmeTr},
the trace  $\theta \in B^*$ is amenable.

\begin{lemma}\label{LemmaQuotient}
The Hilbert spaces $\mathcal{H}_\theta$ and $L^2(\Gamma, \mu \circ \nu)$
are isometric by the correspondence
$U \colon \pi_\theta(\Phi(f)) \xi_\theta \mapsto 
\lambda(f)\psi_0, \ f \in C_c(\Gamma)$.
The equality 
$U \pi_\theta(\Phi(f)) U^* 
= 
\lambda (f)$ holds.
\end{lemma}

\begin{proof}
Note that the state $\theta$ satisfies
\begin{eqnarray*}
\theta(\Phi(f))
=
\theta(E(\Phi(f)))
=
\theta \circ \Phi(f |_{\Gamma(0)})
=
\int_{\Gamma(0)} f|_{\Gamma(0)} \ d \mu,
\quad f \in C_c(\Gamma).
\end{eqnarray*}
For $f_1, f_2 \in C_c(\Gamma)$, we compute
\begin{eqnarray*}
\langle \pi_\theta(\Phi(f_1)) \xi_\theta, 
\pi_\theta(\Phi(f_2)) \xi_\theta \rangle_{\mathcal{H}_\theta}
=
\theta(\Phi(f_2^* f_1))
=
\int_{\Gamma(0)} (f_2^* f_1)|_{\Gamma(0)} \ d \mu
=
\langle f_1, f_2 \rangle_{L^2(\Gamma, \mu \circ \nu)}.
\end{eqnarray*}
Since $\{\pi_\theta(\Phi(f)) \xi_\theta\} \subseteq \mathcal{H}_\theta$ and 
$\{f \} \subseteq L^2(\Gamma, \mu \circ \nu)$
span dense subspaces, $U$ defines a unitary operator.
For $f, f_1, f_2 \in C_c(\Gamma)$, we have
\begin{eqnarray*}
\langle \pi_\theta(\Phi(f)) U^* f_1, 
                       U^* f_2 \rangle_{L^2(\Gamma, \mu \circ \nu)}
&=& \langle \pi_\theta(\Phi(f_2^* f f_1)) \xi_\theta, 
          \xi_\theta \rangle_{\mathcal{H}_\theta} \\
&=& \theta(\Phi(f_2^* f f_1))\\
&=& \int_{\Gamma(0)} (f_2^* f f_1)|_{\Gamma(0)} \ d \mu \\
&=& \langle \lambda(f) f_1, 
                     f_2 \rangle_{L^2(\Gamma, \mu \circ \nu)}.
\end{eqnarray*}
It follows that $U \pi_\theta(\Phi(f)) U^* 
= 
\lambda (f)$.
\end{proof}

By Lemma \ref{LemmaQuotient}, we obtain the identification 
of $\pi_\theta(B)$ and $C^*_{\mathrm{red}}(\Gamma, \nu, \mu)$.

\begin{proposition}\label{PropositionAmenableTrace}
The trace $\tau$ of $C^*_{\mathrm{red}} (\Gamma, \nu, \mu)$ 
is amenable if $B$ is locally reflexive.
\end{proposition}

\begin{proof}
By the last paragraph of Subsection \ref{SubsectionAmeTr},
the tracial state $\theta$ of $B$ is amenable. The trace $\theta$ equals to 
$\tau \circ \pi_\theta$.
By Proposition \ref{PropositionQuotient}, we obtain amenability of $\tau$.
\end{proof}
Among several equivalent conditions for amenability of 
$\tau$,
we only use the following:
There exists a unital completely positive map
$\Psi \colon \mathbb{B} (L^2(\Gamma, \mu \circ \nu)) 
\rightarrow W^*(\Gamma, \nu, \mu)$
whose restriction on 
$C^*_\mathrm{red} (\Gamma, \nu, \mu)$
is the identity map.

\subsection{Proof of the main theorem}

For the first half of the proof, we use an idea by Popa \cite[Lemma 4.2]{Popa:NotesOnCartan}. 
Connes, Feldman, and Weiss proved that the hyperfinte factor of type II$_1$ has a unique Cartan subalgebra up to automorphism (\cite{ConnesFeldmanWeiss}). 
Popa gave a proof of the uniqueness theorem,
making use of
hypertrace.
In this paper, the state 
$\tau \circ \Psi 
\in \mathbb{B}(L^2(\Gamma, \mu \circ \nu))^*$ 
plays the role of hypertrace.

\begin{proposition}\label{PropositionLimInf}
Let $X = \bigsqcup_{m = 1}^\infty X_m$ be a generalized box space and let $L = \{ \Delta_X = \phi(0), \phi(1), \cdots, \phi(k)\}$
be label on a controlled set $T$ of $X$.
Let $\Gamma$ be a groupoid in Theorem \ref{TheoremGroupoid}
and let $\mu$ be a limit measure on the units $\Gamma(0)$.
If the trace $\tau$ of 
$C^*_\mathrm{red} (\Gamma, \nu, \mu)$ is amenable, then $X$ and $T$ satisfy 
the following condition:
for every positive number $\epsilon > 0$, there exists a controlled set $F$,
\begin{eqnarray*}
\liminf_{m \rightarrow \infty} \left(
\inf \left\{ \frac{ \sharp(T[Y]) }{ \sharp(Y) } \ \colon\ 
Y \subseteq X_m, Y \textrm{\ is\ an } 
F \textrm{-bounded\ set} \right\} \right) < 1 + \epsilon.
\end{eqnarray*}
\end{proposition}

\begin{proof}
For 
$i= 0, 1, \cdots, k$,
define a homeomorphism $\widetilde{\psi_i}$ by
the product of $\Gamma(i)$
\begin{eqnarray*}
\widetilde{\psi_i} 
\colon
t^{-1}(s(\Gamma(i)))
\longrightarrow
t^{-1}(t(\Gamma(i)))
\colon
\gamma
\mapsto
\gamma^\prime \gamma,
\end{eqnarray*}
where $\gamma^\prime$ is the unique element of 
$s^{-1}(\{t(\gamma)\}) \cap \Gamma(i)$. 
The map $\widetilde{\psi_i}$ is well-defined and homeomorphic, by condition (\ref{TheoremGroupoidInverseImagek})
in Theorem \ref{TheoremGroupoid}.
A Borel subset 
$\Lambda \subseteq \Gamma$ is said to be an
$\epsilon$-F\o lner set with respect to 
$\{ \widetilde{\psi_i}\}$ if
$0 < \mu \circ \nu(\Lambda) < \infty$ and
\begin{eqnarray*}
\sum_{i = 0}^k \mu \circ \nu 
\left(
\bigcup_{i = 0}^k 
\widetilde{\psi_i} 
\left(
\Lambda
\right) 
\setminus
\Lambda
\right)
< \epsilon\ \mu \circ \nu(\Lambda).
\end{eqnarray*}
We first prove that there exists an $\epsilon$-F\o lner set of these transformations.  

Since the trace 
$\tau$ on $C^*_\mathrm{red} (\Gamma, \nu, \mu)$ 
is amenable,
there exists a unital completely positive map
$\Psi \colon \mathbb{B} (L^2(\Gamma, \mu \circ \nu)) 
\rightarrow W^*(\Gamma, \nu, \mu)$
whose restriction on 
$C^*_\mathrm{red} (\Gamma, \nu, \mu)$
is the identity map.
It follows that the map $\Psi$ has the $C^*_\mathrm{red} (\Gamma, \nu, \mu)$-bimodule property.
Define a linear functional $\rho$ on 
$L^\infty(\Gamma, \mu \circ \nu)$ by the composition of 
$\tau$ and $\Psi |_{L^\infty(\Gamma, \mu \circ \nu)}$, 
namely,
$\rho(\zeta) = \langle \Psi(\zeta) \psi_0, \psi_0 \rangle _{L^2(\Gamma, \mu \circ \nu)}$. 
By the bimodule property of $\Psi$ and the trace property of $\tau$, for every $\zeta \in L^\infty(\Gamma, \mu \circ \nu)$,
we have
\begin{eqnarray*}
\rho(\lambda(\psi_i)^* \zeta \lambda(\psi_i)) 
&=&  
\tau (\lambda(\psi_i)^* \Psi(\zeta) \lambda(\psi_i))\\
&=& 
\tau (\lambda(\psi_i \psi_i^*) \Psi(\zeta))\\
&=& 
\tau (\Psi(\lambda(\psi_i \psi_i^*) \zeta))\\
&=&
\rho(\lambda(\psi_i \psi_i^*) \zeta) .
\end{eqnarray*}
Note 
that 
$\lambda(\psi_i \psi_i^*) \zeta 
\in L^\infty(\Gamma, \mu \circ \nu)$ 
is the restriction of $\zeta$ to 
the range of 
$\widetilde{\psi_i}$ and
that $\lambda(\psi_i)^* \zeta \lambda(\psi_i) \in L^\infty(\Gamma, \mu \circ \nu)$
is a translation of $\lambda(\psi_i \psi_i^*) \zeta$ by $\widetilde{\psi_i}^{-1}$.
The above equality is rephrased as follows:
$\rho ( \zeta \circ \widetilde{\psi_i} ) 
=
\rho(\zeta),\ 
\zeta \in L^\infty ( \mathrm{Image} \, \widetilde{\psi_i} )$.

Take a net $\{ \eta_j \} \subseteq L^1(\Gamma, \mu \circ \nu)$ of positive $L^1$-functions with norm $1$ 
which converges to $\rho \in L^\infty(\Gamma, \mu \circ \nu) ^*$ in the weak$^*$-topology. 
For every $\zeta \in L^\infty ( \mathrm{Image}\, \widetilde{\psi_i} )$,
we have
\begin{eqnarray*}
\lim_j \int_{\mathrm{Image}\, \widetilde{\psi_i}}
\left( \eta_j \circ \widetilde{\psi_{i}}^{-1} \right) 
\zeta d \mu \circ \nu
&=&
\lim_j \int_{\mathrm{Dom}\, \widetilde{\psi_i}}
\eta_j \left( \zeta \circ \widetilde{\psi_i} \right) 
d \mu \circ \nu \\
&=&
\rho \left( \zeta \circ \widetilde{\psi_i} \right) \\ 
&=&
\rho(\zeta) \\
&=&
\lim_j \int_{\mathrm{Image}\, \widetilde{\psi_i}}
\eta_j \zeta d \mu \circ \nu.
\end{eqnarray*}
It follows that $\eta_j \circ \widetilde{\psi_i}^{-1} 
- \left( \eta_j \left|_{\mathrm{Image}\, \widetilde{\psi_i}} \right.\right) \in L^1\left( \mathrm{Image}\, \widetilde{\psi_i} \right)$ converges to $0$ 
in the weak topology for every $i \in I$.  
By the Hahn--Banach theorem, for an arbitrary positive number $\epsilon$, there exists a convex combination $\eta$ of $\{ \eta_j \}$ such that 
\begin{eqnarray*}
\sum_{i = 0}^k 
\left\| \eta \circ \widetilde{\psi_i}^{-1} - 
\left(\eta \left|_{\mathrm{Image}\, \widetilde{\psi_i}} \right.\right) \right\| _{L^1\left( \mathrm{Image}\, \widetilde{\psi_i} \right)}
< \epsilon = \epsilon \| \eta \|_{L^1(\Gamma, \mu \circ \nu)}.
\end{eqnarray*}
It follows that there exists a level set $\Lambda$ of $\eta$ satisfying that
\begin{eqnarray*}
\sum_{i = 0}^k 
\mu \circ \nu \left(\widetilde{\psi_i} \left(
\Lambda 
\right) 
 \bigtriangleup \left(\Lambda \cap \mathrm{Image}\, \widetilde{\psi_i} \right) \right)
< \epsilon \mu \circ \nu (\Lambda).
\end{eqnarray*}
By the inclusion
\begin{eqnarray*}
\bigcup_{i = 0}^k \widetilde{\psi_i} 
\left(
\Lambda 
\right) 
\setminus
\Lambda
\subseteq
\bigcup_{i = 0}^k 
\widetilde{\psi_i} \left(
\Lambda
\right) 
\bigtriangleup 
\left( \Lambda \cap \mathrm{Image}\, \widetilde{\psi_i} \right),
\end{eqnarray*}
we conclude that $\Lambda$ is an $\epsilon$-F\o lner set.

By conditions (\ref{TheoremGroupoidClopen}) and (\ref{TheoremGroupoidGenerator})
in Theorem \ref{TheoremGroupoid}, the topology of $\Gamma$ is generated by countably many compact open subsets.
We may assume that the F\o lner set $\Lambda$ is compact and open.
On the compact subset $\bigcup_{i = 0}^k \widetilde{\psi_i} 
\left( \Lambda \right)$, 
the measure $\mu \circ \nu$ is an accumulation point of 
$\{\mu_m \circ \nu\}_{m \in \mathbb{N}}$. 
Since the characteristic function of $\bigcup_{i = 0}^k \widetilde{\psi_i} \left( \Lambda \right) \setminus \Lambda$ is continuous, 
there exist infinitely many $m(l)$ satisfying the inequality
\begin{eqnarray*}
\mu_{m(l)} \circ \nu \left(
\bigcup_{i = 0}^k \widetilde{\psi_i} 
\left(
\Lambda 
\right) 
\setminus
\Lambda 
\right)
< 
\epsilon\ \mu_{m(l)} \circ \nu (\Lambda).&
\end{eqnarray*}

Let $F \subseteq (X, T)^{(2)}$ be the inverse image of $\Lambda$ with respect to the map $\Phi^*$.
By Lemma \ref{LemmaInverseImage},
the inverse image of
$\bigcup_{i = 0}^k \widetilde{\psi_i} (\Lambda) \setminus \Lambda$ with respect to $\Phi^*$ is 
\begin{eqnarray*}
\bigcup_i  \phi(i) \circ F \setminus F = T \circ F \setminus F.
\end{eqnarray*}
Let $w_{m(l)}$ be the normalized counting measure 
on $X_{m(l)}$. 
Since the measure $\mu_{m(l)} \circ \nu$ is the push forward of $w_{m(l)} \circ c$, we have
$w_{m(l)} \circ c (T \circ F \setminus F)
< 
\epsilon\ w_{m(l)} \circ c (F)$.
By multiplying $\sharp(X_{m(l)})$, we also have
\begin{eqnarray*}
\sharp 
((T \circ F \setminus F)  \cap X_{m(l)}^2) 
< 
\epsilon\ \sharp (F \cap X_{m(l)}^2).
\end{eqnarray*}
By the equalities
\begin{eqnarray*}
F \cap X_{m(l)}^2 &=& 
\bigsqcup_{y \in X_{m(l)}} F[y] \times \{y\},\\
(T \circ F \setminus F)  \cap X_{m(l)}^2 
&=& \bigsqcup_{y \in X_{m(l)}} 
\left( T [F[y]] 
\setminus F[y] \right) \times \{y\},
\end{eqnarray*}
there exists an $F$-ball $Y_{m(l)} = F[y]$ satisfying the following F{\o}lner condition:
\begin{eqnarray*}
\sharp \left(
T[Y_{m(l)}] \setminus Y_{m(l)}
\right)
< 
\epsilon\ \sharp (Y_{m(l)}).
\end{eqnarray*}
Thus we get the following inequality:
\begin{eqnarray*}
\liminf_{m \rightarrow \infty} \left(
\inf \left\{ \frac{ \sharp(T[Y]) }{ \sharp(Y) } \ \colon\ 
Y \subseteq X_m, Y \textrm{\ is\ an } 
F \textrm{-bounded\ set} \right\} \right) 
< 1 + \epsilon.
\end{eqnarray*}
Since $\Lambda$ is a compact open subset of $\Gamma$, there exists $n \in \mathbb{N}$ such that $\Lambda \subseteq \bigcup_{g \in I^n} \Gamma(g)$.
Taking the inverse images with respect to $\Phi^*$, we have $F \subseteq (T \cup T^{-1})^{\circ n}$. 
It follows that the subset $F \subseteq X^2$ is controlled.
\end{proof}

\begin{theorem}\label{TheoremLocalReflexivity}
If $X = \bigsqcup_{m = 1}^\infty X_m$ 
be a sequence of weak expander spaces,
then the uniform Roe algebra $C^*_\mathrm{u}(X)$ is not locally reflexive.
\end{theorem}

\begin{proof}
Suppose that the uniform Roe algebra $C^*_\mathrm{u}(X)$ is locally reflexive.
Take a controlled set $T$ including the diagonal subset and fix a label $\{\Delta_X = \phi(0), \phi(1), \cdots, \phi(k) \}$ on $T$.
Take a groupoid $\Gamma$ in Theorem \ref{TheoremGroupoid}. The translation C$^*$-algebra $B$ associated to $\Gamma$ is also locally reflexive, 
since $C^*_\mathrm{u}(X)$ is locally reflexive.
Let $\mu$ be a limit measure on
$\Gamma(0)$ constructed in Subsection \ref{SubsectionMeasuredGroupoid}. 
By Proposition \ref{PropositionAmenableTrace}, 
the corresponding trace of the reduced C$^*$-algebra $C^*_\mathrm{red}(\Gamma, \nu, \mu)$ is amenable. 
By Proposition \ref{PropositionLimInf}, the controlled set $T$ satisfies the following:
for every positive number $\epsilon > 0$, there exists a controlled set $F$ such that
\begin{eqnarray*}
\liminf_{m \rightarrow \infty} \left(
\inf \left\{ \frac{ \sharp(T[Y]) }{ \sharp(Y) } \ \colon\ 
Y \subseteq X_m, Y \textrm{\ is\ a } 
F \textrm{-bounded\ set} \right\} \right) < 1 + \epsilon.
\end{eqnarray*}
This condition holds true for every controlled set $T$.
It follows that $X$ is not a sequence of weak expander spaces. 
\end{proof}

For a weak expander sequence 
$X = \bigsqcup_{m = 1}^\infty X_m$ consisting of finite metric spaces, we can define a metric $d$ on $X$ so that
$d(X_m, X_n)$ goes to infinity as $m, n \rightarrow \infty$.
The coarse structure of the metric space $(X, d)$ is larger than that of the generalized box space $X$.
Since the uniform Roe algebra $C^*_\mathrm{u}(X, d)$
of the metric space contains $C^*_\mathrm{u}(X)$,
the algebra in Theorem \ref{TheoremLocalReflexivity} can be replaced with $C^*_\mathrm{u}(X, d)$.

\section{Problems related to uniform local amenability}
\label{SectionULA}
Brodzki, Niblo, \v{S}pakula, Willett, and Wright introduced two kinds of uniform local amenability in their study of property A.
The following is related to weak expander spaces.

\begin{definition}[Definition 2.2 of \cite{BNSWW}]
Let $X$ be a uniformly locally finite metric space.
The space $X$ is said to have property ULA
if for every $\epsilon > 0$ and $R > 0$, 
there exists $S > 0$ satisfying the following condition:
For any finite subset $W$ of $X$, there
exists $Y \subseteq X$ such that $\mathrm{diam}(Y) \le S$ and
$\sharp(\partial_R Y \cap W) < \epsilon \sharp(Y)$
\end{definition}

The property ULA is 
the weaker form of uniform local amenability.
The stronger one is called ULA$_\mu$. 
The property ULA$_\mu$ is equivalent to property A
(Combination of \cite{BNSWW} and \cite[Theorem 4.1]{SakoONLP}). We note that
$\ \textrm{Property A} \Leftrightarrow 
\textrm{ULA}_\mu \Rightarrow
\textrm{ULA}$.

\begin{proposition}
Let $(X, d)$ be a uniformly locally finite metric space.
The space $X$ does not have property ULA
if and only if
there exists a sequence of finite subsets 
$X_m \subseteq X$ whose disjoint union $\bigsqcup_{m = 1}^\infty X_m$ is a sequence of weak expander spaces.
\end{proposition}

\begin{proof}
Suppose that $X$ does not have property ULA.
Then there exist $c > 0$ and $R > 0$ with the following condition: 
for every natural number $m$, there exists a finite subset $X_m \subseteq X$ 
such that every subset $Y \subseteq X_m$ with $\mathrm{diam}(Y) < m$ satisfies the inequality
$\sharp(\partial_R (Y) \cap X_m) \ge c \sharp(Y)$.
In other words, the space $(X_m, d)$ satisfies
\begin{eqnarray*}
\mathrm{inf} \left\{ \left. 
\frac{\sharp(N_R (Y) \cap X_m)}{\sharp(Y)} 
\right| Y \subseteq X_m, \mathrm{diam}(Y) \le m
\right\}
\ge 1 + c. 
\end{eqnarray*}
It follows that $\bigsqcup_{m = 1}^\infty X_m$ is a sequence of weak expander spaces.

Conversely, we suppose that there exists a sequence of finite subspaces $X_m \subseteq X$ 
such that $\bigsqcup_{m = 1}^\infty X_m$ is a sequence of weak expander spaces.
Then there exist $c > 0$ and $R > 0$ such that the following holds for every $S > 0$: 
for large enough $m \ge M_S$,
\begin{eqnarray*}
\mathrm{inf} \left\{ \left. 
\frac{\sharp(N_R (Y) \cap X_m)}{\sharp(Y)} 
\right| Y \subseteq X_m, \mathrm{diam}(Y) < S
\right\}
\ge 1 + c. 
\end{eqnarray*}
For $m \ge M_S$, there exists no finite subset $Y \subseteq X_m$ such that $\mathrm{diam}(Y) < S$ and $\sharp(\partial_R Y \cap X_m) < c \sharp(Y)$. 
It follows that $X$ does not have ULA.
\end{proof}

\begin{corollary}\label{CorollaryNotA}
Let $(X, d)$ be a uniformly locally finite metric space.
The space $X$ does not have property A if
there exists a sequence of finite subsets 
$X_m \subseteq X$ such that the disjoint union $\bigsqcup_{m = 1}^\infty X_m$ is a sequence of weak expander spaces.
\end{corollary}

\begin{proof}
Property A implies property ULA (\cite[Lemma 2.7 and Proposition 3.2]{BNSWW}).
\end{proof}

We close this paper by proposing three problems.
Let $X$ be a uniformly locally finite metric space.
\begin{problem}\label{ProblemDisjoint}
Suppose that $X$ does not have property ULA.
Can one find \underline{disjoint} finite subsets 
$X_m \subseteq X$ 
such that $\bigsqcup_{m = 1}^\infty X_m$ is a sequence of weak expander spaces?
\end{problem}

\begin{problem}\label{ProblemExact}
Let $\{X_m\}_{m \in \mathbb{N}}$ be copies of $X$.
Suppose that $C^*_\mathrm{u}(X)$ is exact (resp.\ locally reflexive). Is 
$C^*_\mathrm{u}(\bigsqcup X_m)$ also exact (resp.\ locally reflexive)?
\end{problem}

If the answer of Problem \ref{ProblemDisjoint} or Problem \ref{ProblemExact} is affirmative, so is Problem \ref{ProblemExactImpliesULA}.

\begin{problem}\label{ProblemExactImpliesULA}
Suppose that 
the uniform Roe algebra $C^*_\mathrm{u}(X)$ is locally reflexive.
Does the space $X$ have property ULA? 
What about the case that $C^*_\mathrm{u}(X)$ is exact?
\end{problem}

In the case that $X$ coarsely embeds into a discrete group, Brodzki, Niblo, and Wright 
proved in \cite{PaperBrodzkiNiblo} that exactness of $C^*_\mathrm{u}(X)$ implies property A of $X$.

\bibliographystyle{amsalpha}
\bibliography{wexpander.bib}

\end{document}